\documentclass[12pt]{article}
\usepackage[utf8]{inputenc}
\usepackage{tikz}
\usepackage{amsthm}
\usepackage{amsmath}
\usepackage{amssymb}
\usepackage{amsthm}
\usepackage{amscd}
\usepackage{tikz-cd}
\usepackage[all]{xy}
\usepackage[margin=1in]{geometry}
\usepackage{comment}
\usepackage{pdfpages}
\usepackage{url}
\usepackage[T1]{fontenc}
\usepackage{amsmath,amsfonts,amssymb,enumerate}
\usepackage{amscd}
\usepackage{graphicx}
\usepackage[T1]{fontenc}
\usepackage{amsmath,amsfonts,amssymb,enumerate}
\usepackage{amsthm} 
\usepackage[english]{babel}
\newcommand{\ec}{\color{black}} 

\newtheorem{thm}{Theorem}[section]
\newtheorem{lem}[thm]{Lemma}
\newtheorem{prop}[thm]{Proposition}
\newtheorem{cor}[thm]{Corollary}

\newtheorem{conjecture}[thm]{Conjecture}

\theoremstyle{definition}

\newtheorem{definition}[thm]{Definition}
\newtheorem{remark}[thm]{Remark}

\numberwithin{equation}{section}

\newcommand{\FF}{\mathbb{F}}
\newcommand{\A}{\mathbb{A}}
\newcommand{\BB}{\mathbb{B}}
\newcommand{\DD}{\mathbb{D}}

\newcommand{\m}{\mathfrak{m}}
\def\bw{\bigwedge}

\def\ker{\mbox{\rm ker}}

\def\rank{\mbox{\rm rank}}

\usepackage{array}

\newcommand{\ra}{\rightarrow}

\begin{document}

\title{Mapping free resolutions of length three II - Module formats}

\author{Sara Angela Filippini \thanks{ Universit\`a del Salento, Dipartimento di Matematica e Fisica ``Ennio De Giorgi'', Lecce, \textbf{Email address:} saraangela.filippini@unisalento.it  
} 
\and Lorenzo Guerrieri \thanks{ Jagiellonian University, Instytut Matematyki, Krak\'{o}w, \textbf{Email address:} lorenzo.guerrieri@uj.edu.pl}
} 
\maketitle
\begin{abstract}
\noindent
Let $M$ be a perfect module of projective dimension 3 over a Gorenstein, local or graded ring $R$. We denote by $\FF$ the minimal free resolution of $M$. Using the generic ring associated to the format of $\FF$ we define higher structure maps, according to the theory developed by Weyman in \cite{W18}. We introduce a generalization of classical linkage for $R$-module using the Buchsbaum--Rim complex, and study the behaviour of structure maps under this Buchsbaum--Rim linkage. In particular, for certain formats we obtain criteria for these $R$-modules to lie in the Buchsbaum--Rim linkage class of a Buchsbaum--Rim complex of length 3. \\

\end{abstract}

\section{Introduction}
Free resolutions of ideals and modules have been investigated for a long time. Given a module $M$ over a commutative local or graded ring $R$, admitting a minimal finite free resolution $\FF$, we define its {\it format} as the sequence of the ranks of the free modules of $\FF$ (Betti numbers). An important task is to classify ideals and modules having minimal free resolution of a given format. The classical examples which gave rise to the subject are Hilbert--Burch and Buchsbaum--Eisenbud structure theorems, classifying perfect ideals of height 2 and Gorenstein ideals of height 3, respectively. A related widely studied problem is the classification of perfect ideals which are in the linkage class of a complete intersection ({\it licci} ideals). Perfect ideals of height 2 and Gorenstein ideals of height 3 are well-known examples of licci ideals (\cite{CVWlinkage,HU,PS74, Watanabe}). \\
This paper is motivated by previous work of Weyman and other authors to understand in general the structure of resolutions of modules of projective dimension $3$ \cite{CKLW,CVWdynkin,PW90,w89}. A helpful tool for this task is provided by the generic ring and generic complex associated to each given format. In \cite{W18} a generic ring $\hat{R}_{gen}$ is constructed for resolutions of any format of length 3 by using methods from representation theory. In particular, there is an action of a Kac--Moody Lie algebra on ${\hat R}_{gen}$, whose properties are fundamental to describe the modules having free resolution of that format. Moreover, the ring $\hat{R}_{gen}$ is Noetherian if and only if this Lie algebra is finite-dimensional. This condition singles out a nice collection of formats, called {\it Dynkin} formats, since in the finite-dimensional case the Lie algebra corresponds to a classical Dynkin diagram.
In the case of cyclic modules of projective dimension 3 (i.e. $R/I$ for an ideal $I$) we have the following Dynkin formats:
\begin{itemize}
\item $A_n$: $(1,3,n+2,n)$ for $n\geq 1$;
\item $D_n$: $(1,n, n,1)$ and $(1,4,n,n-3)$ for $n\geq 4$;
\item $E_6$: $(1,5,6,2)$;
\item $E_7$: $(1,5,7,3)$ and $(1,6,7,2)$;
\item $E_8$: $(1,5,8,4)$ and $(1,7,8,2)$.
\end{itemize} 
It is conjectured that perfect ideals having minimal free resolution of Dynkin format can be obtained as specialization of the defining ideals of Schubert varieties \cite{SW21}, and that they are all licci \cite{CVWdynkin}. Furthermore, in \cite{CVWdynkin} it is proved that there exist non-licci ideals with minimal resolution of any non-Dynkin format. Finite free resolutions of ideals of Schubert varieties in exceptional minuscule homogeneous spaces were investigated in \cite{FTW23}.

By studying the action of the associated Lie algebra, Weyman observed that the differentials and the multiplicative structure  of a resolution (which is well-known since the famous Buchsbaum--Eisenbud papers \cite{BE74}, \cite{BE77}) are only the first steps of a more complicated collection of linear maps. These linear maps are called {\it higher structure maps} and are induced by three special representations of the Lie algebra, called {\it critical representations}. In \cite{GW20} the higher structure maps are described for formats $(1,n,n,1)$ and $(1,4,n,n-3)$. In \cite{GNW} the authors study how certain higher structure maps behave under linkage, and give a criterion for ideals admitting a resolution of format $(1,5,6,2)$ to be licci.

This paper aims to study non-cyclic modules having minimal free resolutions of length 3 for small Dynkin formats, following the methods used in \cite{GW20} and \cite{GNW}. We focus on the formats $(2,4,4,2)$, $(2,5,4,1)$, $(2,5,5,2)$ and $(2,6,5,1)$. We look at their critical representations and higher structure maps, and we consider their ``linkage'' properties. A module $M$ is called {\it perfect} if its projective dimension equals the depth of its annihilator. Perfect modules with free resolution of format $(2,4,4,2)$ are resolved by a Buchsbaum--Rim complex (\cite{Buchsbaum,BuchsbaumRim}, see also \cite{CAwav}). For a free resolution $\FF$ of this type, we introduce a notion of linkage obtained as the dual of the mapping cone of a complex map from a Buchsbaum--Rim complex (defined in terms of $\FF$) to $\FF$, which we call {\it Buchsbaum--Rim (BR) linkage}.
Other versions of linkage for non-cyclic modules were considered in \cite{Dibaei,Martin,Martsinkovsky,Nagel,Yoshino}.

The paper is organized as follows. In Section \ref{sec:critref} we introduce the relevant notations and provide the formulas of the higher structure maps for the aforementioned formats. We then characterize perfect modules with resolutions of formats $(2,4,4,2)$ and $(2,5,4,1)$. 

In Section \ref{sec:BRlinkage} we introduce the notion of BR linkage and study perfect modules having minimal free resolutions of formats $(2,6,5,1)$ and $(2,5,5,2)$. We investigate how higher structure maps change under this BR linkage. Finally, we provide a criterion for such modules to be in the BR-linkage class of a not necessarily minimal BR complex of format $(2,4,4,2)$.

Lastly, Section \ref{sec:splitex} is devoted to computing some higher structure maps over the split exact complexes of formats $(2,6,5,1)$ and $(2,5,5,2)$. It contains technical results that are needed to complete the proofs in Section \ref{sec:BRlinkage}.

\section{Critical representations and higher structure maps for module formats} \label{sec:critref}

Let $R$ be a commutative Noetherian ring. We generally assume $R$ to be Gorenstein and local or graded, with maximal ideal $\m$ and infinite residue field $K$. We also assume $\frac{1}{2} \in R$. For a matrix $A$ with entries in $R$ we always denote by $I_d(A)$ the ideal generated by its $d\times d$ minors.

We will work with free resolutions of modules over $R$ of the form
\begin{equation}
\label{complexF}
\FF: 0 \longrightarrow F_3 \buildrel{d_3}\over\longrightarrow  F_2 \buildrel{d_2}\over\longrightarrow F_1 \buildrel{d_1}\over\longrightarrow F_0.
\end{equation}
We denote by $r_i$ the rank of $F_i$ and say that the complex $\FF$ has format $(r_0,r_1,r_2,r_3)$. 
 Throughout the paper we assume $F_0 \cong R^2$ and describe the formulas of some structure maps for this kind of free resolutions. 
 We expect similar formulas to hold more in general when $r_0 > 2$. 
 The bases of $F_0, F_1,F_2,F_3$ will be respectively denoted by $\lbrace u_1, u_2 \rbrace$, $\lbrace e_1, \ldots, e_{r_1}\rbrace$, $\lbrace f_1, \ldots, f_{r_2} \rbrace$, $\lbrace g_1, \ldots, g_{r_3}\rbrace$.

In \cite{W18} Weyman constructed a generic ring ${\hat R}_{gen}$ associated to any given format of free resolution of length 3. There is a Kac--Moody Lie algebra acting on ${\hat R}_{gen}$, whose properties are fundamental to describe the modules having free resolution of that format. 

There are three representations $W(d_3), W(d_2), W(d_1)$ of this Lie algebra, called \it critical representations, \rm which are needed to describe the generators of the generic ring and their relations. 
The graded components of these representations correspond to maps involving symmetric powers, exterior powers, and more complicated Schur functors of the modules $F_0, F_1, F_2, F_3$. The zero-graded components of the critical representations correspond to the three differentials $d_3, d_2, d_1$. In the case of free resolutions of cyclic modules, the graded components of degree one correspond to the multiplicative structure of $\FF$. For non-cyclic modules the maps in the degree one components can be computed using the comparison map from a Buchsbaum--Rim complex to the complex $\FF$, analogously as one computes the multiplicative structure for an ideal using a complex map from the Koszul complex. The formulas are discussed in detail in Section 2.1.
All the maps corresponding to components of degree larger than zero are called \it higher structure maps\rm. 

For Dynkin formats the critical representations are finite-dimensional and the higher structure maps have nice applications, describing properties of the generic rings and of the structure of the modules of the given format. 

In this paper we will work mainly with formats $(2,5,5,2)$ and $(2,6,5,1)$.
 For the format $(2,5,5,2)$ the critical representations are:
 $$W(d_3)=F^*_2 \otimes [F_3 \oplus \bigwedge^3 F_1 \oplus \bigwedge^5 F_1 \otimes F_1 \otimes F_3^* \oplus S_{2,2,2,2,1} F_1 \otimes \bigwedge^2 F_3^* ],$$
$$W(d_2)=F_2 \otimes [F_1^* \oplus \bigwedge^2 F_1 \otimes F_3^* \oplus (\bigwedge^4 F_1 \otimes F_1 \otimes \bigwedge^2 F_3^* \oplus \bigwedge^5 F_1 \otimes S_2 F_3^*  ) \oplus  $$
$$ \oplus S_{2,2,2,1,1}F_1 \otimes S_{2,1}F_3^* \oplus  S_{3,2,2,2,2}F_1 \otimes S_{2,2}F_3^* ],  $$
$$W(d_1)= F_0^* \otimes [F_1 \oplus \bigwedge^4 F_1 \otimes F_3^* \oplus \bigwedge^5 F_1 \otimes \bigwedge^2 F_1 \otimes  \bigwedge^2F_3^* \oplus
 \bigwedge^5F_1 \otimes \bigwedge^5F_1 \otimes  S_{2,1}F_3^*].  $$
For the format $(2,6,5,1)$ the critical representations are:
 $$W(d_3)=F^*_2 \otimes [F_3 \oplus \bigwedge^3 F_1 \oplus (\bigwedge^5 F_1 \otimes F_1 \oplus \bigwedge^6 F_1) \otimes F_3^* \oplus $$  $$ \oplus S_{2,2,2,1,1,1} F_1 \otimes S_2F_3^* \oplus S_{2,2,2,2,2,2}F_1 \otimes S_3 F_3^* ],$$
$$W(d_2)=F_2 \otimes [F_1^* \oplus \bigwedge^2 F_1 \otimes F_3^* \oplus \bigwedge^5 F_1 \otimes S_2 F_3^*], $$
$$W(d_1)= F_0^* \otimes [F_1 \oplus \bigwedge^4 F_1 \otimes F_3^* \oplus \bigwedge^6 F_1 \otimes F_1 \otimes  S_2F_3^*].  $$

 For the other Dynkin formats the tables describing the critical representations can be found in \cite{LW19}. 
Formulas for computing explicitly these maps are given in \cite{GW20} for cyclic module formats $(1,n,n,1)$ and $(1,4,n, n-3)$ and in \cite{GNW} for components of small degree (up to four) of arbitrary cyclic module formats (i.e.\ $r_0=1$). The way to compute these structure maps relies on lifting a cycle in some acyclic complex.
In some cases the higher structure maps are computed by choosing generic liftings introducing sets of new variables over $R$, called {\it defect variables}. Indeed, the lift of a cycle may not be unique and the defect variables are used to parametrize generically this non-uniqueness. In this paper we do this operation in the last section, where we work with a split exact complex. In the notation of \cite{GW20} the maps, computed generically using the defect variables, are denoted by $v^{(i)}_j$, where $i=1,2,3$ denotes the critical representation and $j$ denotes the graded component.
Throughout this paper, following the notation of \cite{GNW}, we call $w^{(i)}_{j,k}$ some chosen image of the corresponding map, computed over the ring $R$ without adding new defect variables.
The index $k$ here denotes the fact that for formats different from $D_n$ some graded components involve more than one map. When this is not the case and there is only one map, we simply use the notation $w^{(i)}_{j}$.

\subsection{Definition of certain higher structure maps}
 
We describe how to compute all the structure maps needed in this paper. These maps are computed by lifting a cycle in an exact complex, which is usually associated to $\FF$ or to some Schur complex in the modules $F_0, F_1, F_2, F_3$ (for a treatment of Schur functors and Schur complexes see \cite{ABW}).

We start from the maps in the first graded components $w^{(3)}_1$, $w^{(2)}_1$, $w^{(1)}_1$. 
For these, we first observe that the matrix of $d_1$ is of size $2 \times r_1$. We denote by $m_{ij}$ its $2 \times 2$ minor relative to the columns $i,j$ (with the convention of adding a negative sign if $i > j$).

The map $w^{(3)}_1: \bigwedge^3 F_1 \to F_2$ is defined by lifting the cycle Im($q^{(3)}_1$) in the following complex
\begin{center}
\begin{tikzcd}
0 \arrow[r] & F_3 \arrow[r] & F_2 \arrow[r] & F_1 \arrow[r] & F_0 \\ 
&& \bw^3 F_1 \arrow[u,"w^{(3)}_1"] \arrow[ru,"q^{(3)}_1" ']  \\
\end{tikzcd}
\end{center}
where
\begin{equation}
\label{lift3,1}
q^{(3)}_1(e_i \wedge e_j \wedge e_k):= m_{ij}e_k - m_{ik}e_j + m_{jk}e_i.
\end{equation}
This lift can be interpreted as the comparison map from the Buchsbaum--Rim complex on the map $d_1: F_1 \to F_0$ to the complex $\FF$. This generalizes the procedure commonly used for ideals (cyclic module formats) to compute the multiplication map $\bigwedge^2F_1 \to F_2$, comparing the minimal free resolution with the Koszul complex on a set of minimal generators. 

The image of $q^{(3)}_1$ is in the kernel of $d_1$, and therefore
in the image of $d_2$. Hence it can be lifted to $F_2$.
The lift is not unique, since it can be modified by adding any element in the image of $d_3$ (equal to the kernel of $d_2$). In Section 4, to parametrize generically all the possible liftings we will add a new set of variables as done in \cite{GW20}, \cite{GNW}. 

Also all subsequent maps $w^{(i)}_j$ are defined by lifting the image of an opportune map $q^{(i)}_j$ along an exact complex (as in \cite{GW20}, \cite{GNW} in the case of resolutions of cyclic modules). To check that $q^{(i)}_j$ defines a cycle it is sufficient to show it over a split exact complex, using generic liftings, and  then apply \cite[Theorem 2.1]{GNW} (see Remark \ref{thmweyman} in this paper). In Section \ref{sec:splitex}, we show how to perform the computations over a split exact complex with generic liftings. 

To compute $w^{(2)}_1: \bw^2 F_1 \otimes F_2 \ra F_3$ we consider the map $q^{(2)}_1: \bw^2 F_1 \otimes F_2 \ra F_2$ given by 
\begin{equation}
\label{lift2,1}
q^{(2)}_1(e_i\wedge e_j \otimes f_h) = m_{ij} f_h - w^{(3)}_1 \left(d(f_h)\wedge e_i \wedge e_j\right).
\end{equation}
The image of $q^{(2)}_1$ lies in $\ker(d_2)=\mbox{Im}(d_3)$. Hence, we define $w^{(2)}_1$ as the lift of $q^{(2)}_1$ along the differential $d_3: F_3 \to F_2$. Since $d_3$ is injective, this lift is unique (after fixing a choice of $w^{(3)}_1$).
For the other $w^{(i)}_j$'s we write down the formulas for $q^{(i)}_j$ only for a fixed chosen set of indices in $F_1, F_2$. For all the other possible combinations of indices the terms are defined analogously, respecting the usual skew-symmetric rules of exterior powers and Schur functors. For simplicity we set
$\varepsilon_{1, \ldots,i }:= e_{1} \wedge e_{2} \wedge \ldots \wedge e_{i}$, and let $\varepsilon_{1,\ldots,\hat{j}, \ldots t}$
denote the wedge product of all the elements $e_{1}, \ldots, e_{t}$ distinct from $e_{j}$.

For the map $w^{(1)}_1: \bw^4 F_1 \ra F_0\otimes F_3$, we define $q^{(1)}_1:\bw^4 F_1 \ra F_0\otimes F_2$ as
\begin{equation}
\label{lift1,1}
q^{(1)}_1(\varepsilon_{1,2,3,4}) = \sum_{j=1}^4 (-1)^{j+1} d_1(e_{j}) \otimes w^{(3)}_1 (\varepsilon_{1,\ldots,\hat{j}, \ldots 4}).
\end{equation}
The image of $q^{(1)}_1 $ lies in the kernel of $1_{F_0}\otimes d_2$, and therefore can be lifted via $1_{F_0}\otimes d_3$ to $F_0 \otimes F_3$.

We give now formulas for maps in the second graded components. Since we aim to work out the linkage properties of small formats, we reduce our description to specific cases.
  
In $W(d_3)$ we find two maps $w^{(3)}_{2,1}: \bigwedge^5 F_1 \otimes F_1 \to F_2 \otimes F_3$ and $w^{(3)}_{2,2}: \bigwedge^6 F_1 \to F_2 \otimes F_3$. Both these maps are defined by lifting a cycle $\mbox{Im}(q^{(3)}_{2,k}) \in S_2F_2$ in 
the following complex:
$$ 0 \longrightarrow \bigwedge^2F_3  \longrightarrow F_3\otimes F_2  \longrightarrow  S_2 F_2  \longrightarrow S_2F_1. $$
This fact can be proved explicitly applying the map $S_2F_2 \to S_2F_1$ to the image of $q^{(3)}_{2,k}$ and checking that it is zero by using Pl\"{u}cker relations.
We first look at $w^{(3)}_{2,1}$. 
Adopting the notation $e_i^.e_j^.e_k$ for $w^{(3)}_{1}(e_i \wedge e_j \wedge e_k)$, we set
\begin{equation}
\label{lift3,21}
q^{(3)}_{2,1}(\varepsilon_{1,\ldots,5} \otimes e_1):= e_1^.e_2^.e_3 \otimes e_1^.e_4^.e_5 - e_1^.e_2^.e_4 \otimes e_1^.e_3^.e_5 + e_1^.e_2^.e_5 \otimes e_1^.e_3^.e_4,
\end{equation}
seen as an element of the symmetric power $S_2F_2$.
If $r_1\geq 6$, we set 
\begin{equation}
\label{lift3,22}
q^{(3)}_{2,2}(\varepsilon_{1,\ldots,6}):= \sum_{1 \leq i < j}^{5} (-1)^{i+j+1} e_i^.e_j^.e_6 \otimes \varepsilon_{1,\ldots,\hat{i},\hat{j}, \ldots 5}.
\end{equation}
Notice that both $ q^{(3)}_{2,1} $ and $ q^{(3)}_{2,2} $ composed with the map $d_2 \otimes d_2: S_2F_2 \to S_2F_1$ are zero. Therefore their images lift to $F_3 \otimes F_2$. If $r_3 \geq 2$, this lift is not unique and one can add a second set of defect variables to parametrize it generically. \ec

Let us now define two maps $w^{(1)}_{2,1}: \bigwedge^5 F_1 \otimes \bw^2 F_1 \to F_0 \otimes \bw^2 F_3$ and $w^{(1)}_{2,2}: \bigwedge^6 F_1 \otimes F_1 \to F_0 \otimes S_2F_3$ arising from the representation $W(d_1)$.

If $r_1=5$, we obtain $w^{(1)}_{2,1}$ as lift of $q^{(1)}_{2,1}: \bigwedge^5 F_1 \otimes \bw^2 F_1 \to F_2 \otimes F_3 \otimes F_0$ along the map $\bw^2 F_3 \otimes F_0 \to F_3 \otimes F_2 \otimes F_0$ induced by $d_3 \otimes id_{F_0}$ (the element $a \wedge b \otimes u$ is mapped to $[d_3(a) \otimes b  - a \otimes d_3(b)]\otimes u$). This is defined as
\begin{equation}
\label{lift1,21}
  q^{(1)}_{2,1}(\varepsilon_{1,\ldots,5} \otimes e_1 \wedge e_2):= \sum_{i=3}^{5} (-1)^{i+1} e_1^.e_2^.e_i \otimes w^{(1)}_1( \varepsilon_{1,\ldots,\hat{i}, \ldots, 5}) + \sum_{j=1}^2 (-1)^j d_1(e_j) \otimes w_{2,1}^{(3)}(\varepsilon_{1,\ldots,5} \otimes e_j).
 \end{equation}

For $r_1=6$, the map $w^{(1)}_{2,2}$ is the lift of $q^{(1)}_{2,2}: \bigwedge^6 F_1 \otimes F_1 \to F_2 \otimes F_3 \otimes F_0$. This is defined as
\begin{equation}
\label{lift1,22}
q^{(1)}_{2,2}(\varepsilon_{1,\ldots,6} \otimes e_1):= \sum_{2 \leq i < j}^{6} (-1)^{i+j+1} e_1^.e_i^.e_j \otimes w^{(1)}_1( \varepsilon_{1,\ldots,\hat{i}, \hat{j}, \ldots, 6}) + d_1(e_1) \otimes w_{2,2}^{(3)}(\varepsilon_{1,\ldots,6}).
\end{equation}

Let us look at two maps from the second graded component of $W(d_2)$. These are $w^{(2)}_{2,1}: \bigwedge^4 F_1 \otimes F_1 \otimes F_2 \to \bw^2 F_3$ and $w^{(2)}_{2,2}: \bigwedge^5 F_1 \otimes F_2 \to S_2F_3$. Again we define maps $ q^{(2)}_{2,1}, q^{(2)}_{2,2} $ having target $F_3 \otimes F_2$ and we lift respectively to $ \bw^2 F_3 $ or to $ S_2 F_3 $, as for the two maps defined above in $W(d_1)$. For $q^{(2)}_{2,1}$ we give two different formulas depending on the configuration of the basis elements of $ \bigwedge^4 F_1 \otimes F_1 $ (such formulas are equivalent up to multiplying by a constant):
\begin{equation}
\label{lift2,21+}
 q^{(2)}_{2,1}(\varepsilon_{1,2,3,4} \otimes e_1 \otimes f_h):= e_1^.e_2^.e_3 \otimes w^{(2)}_1(e_1 \wedge e_4 \otimes f_h)-e_1^.e_2^.e_4 \otimes w^{(2)}_1(e_1 \wedge e_3 \otimes f_h)+ $$
 $$ + e_1^.e_3^.e_4 \otimes w^{(2)}_1(e_1 \wedge e_2 \otimes f_h)+ f_h \otimes d_1(e_1) \wedge w^{(1)}_1(\varepsilon_{1,2,3,4})-w^{(3)}_{2,1}(\varepsilon_{1,2,3,4} \wedge d_2(f_h) \otimes e_1),
 \end{equation}
 
\begin{equation}
\label{lift2,21}
q^{(2)}_{2,1}(\varepsilon_{1,2,3,4} \otimes e_5 \otimes f_h):= \sum_{1 \leq i < j \leq 4} (-1)^{i+j+1} w^{(2)}_1(\varepsilon_{1, \ldots, \hat{i}, \hat{j}, \ldots, 4} \otimes f_h) \otimes e_i^.e_j^.e_5 +
$$
$$ + \sum_{1 \leq i \leq 4} (-1)^{i} w^{(2)}_1(e_i \wedge e_5 \otimes f_h) \otimes w^{(3)}_1(\varepsilon_{1, \ldots, \hat{i}, \ldots, 4 }) + \frac{1}{2}f_h \otimes \sum_{1 \leq i \leq 5} (-1)^{i} w^{(1)}_1(\varepsilon_{1, \ldots, \hat{i}, \ldots, 5 }) \wedge d_1(e_i) +   $$
 $$ -w^{(3)}_{2,1}(\varepsilon_{1,2,3,4} \wedge d_2(f_h) \otimes e_5),
\end{equation}

\begin{equation}
\label{lift2,22}
 q^{(2)}_{2,2}(\varepsilon_{1,\ldots,5} \otimes f_h):= \sum_{i,j,k} (-1)^{i+j+k} e_i^.e_j^.e_k  \otimes w^{(2)}_1(\varepsilon_{1,\ldots,\hat{i}, \hat{j}, \hat{k}, \ldots, 5} \otimes f_h) + $$ $$ +  \frac{1}{2}f_h \otimes \sum_{i} (-1)^i d_1(e_i) \wedge w^{(1)}_1(\varepsilon_{1,\ldots,\hat{i}, \ldots, 5}) - w_{2,2}^{(3)}(\varepsilon_{1,\ldots,5} \wedge d_2(f_h)).
\end{equation}

Finally, we describe one map from the third graded components of $W(d_1)$ if $r_1=5$. 
This map is $w^{(1)}_{3}: \bigwedge^5 F_1 \otimes \bw^5 F_1 \to F_0 \otimes \bw^2 F_3 \otimes F_3$ and can be obtained as lift of the image of $q^{(1)}_{3}: \bigwedge^5 F_1 \otimes \bw^5 F_1 \to F_0 \otimes F_3 \otimes F_3 \otimes F_2$,
along the map obtained as tensor product of $id_{F_0}$ with the map $\bigwedge^2 F_3 \otimes F_3 \to F_3 \otimes F_3 \otimes F_2$ sending $a \wedge b \otimes c$ to $[d_3(a) \otimes b  - a \otimes d_3(b)] \otimes c$. Let us define
\begin{equation}
\label{lift1,3}
q^{(1)}_{3}(\varepsilon_{1,\ldots,5} \otimes \varepsilon_{1,\ldots,5}):= \sum_{i=1}^5 (-1)^{i+1} w^{(1)}_1(\varepsilon_{1,\ldots,\hat{i}, \ldots, 5}) \otimes w^{(3)}_{2,1}(\varepsilon_{1,\ldots, 5} \otimes e_i).
\end{equation}

\subsection{Perfect modules with resolution of format $(2,4,4,2)$}

We now discuss the structure of perfect modules with resolution of format $(2,4,4,2)$. The critical representations for this format are:
$$ W(d_3)= F_2^* \otimes [ F_3 \oplus \bw^3 F_1 ], $$
$$ W(d_2)  =  F_2 \otimes[ F_1^* \oplus  \bw^2 F_1 \otimes F_3^* \oplus \bigwedge^4 F_1 \otimes F_1 \otimes \bw^2 F_3^*], $$
$$ W(d_1)  =    F_0^* \otimes [ F_1 \oplus \bw^4 F_1 \otimes F_3^* ]. $$
Perfect modules with resolution of format $(2,4,4,2)$ are resolved by a Buchsbaum--Rim complex. Let $M$ be one of such modules defined over the ring $R$. The free resolution of $M$ is given by 
\begin{equation}
\label{buchsbaumrymcomplex}
\FF : 0 \longrightarrow F_3 \buildrel{d_3}\over\longrightarrow  F_2 \buildrel{d_2}\over\longrightarrow F_1 \buildrel{d_1}\over\longrightarrow F_0 \longrightarrow M,
\end{equation}
where the differentials are
$$d_3 =\bmatrix x_{11}  & x_{21} \\ x_{12} & x_{22} \\ x_{13} & x_{23} \\ x_{14} & x_{24} \endbmatrix , \quad 
d_2=\bmatrix 0 & X_{\hat{1}\hat{2}} & -X_{\hat{1}\hat{3}} & X_{\hat{1}\hat{4}} \\  -X_{\hat{1}\hat{2}} & 0 &  X_{\hat{2}\hat{3}}& -X_{\hat{2}\hat{4}}\\ X_{\hat{1}\hat{3}} & -X_{\hat{2}\hat{3}} &  0 & X_{\hat{3}\hat{4}} \\ -X_{\hat{1}\hat{4}} & X_{\hat{2}\hat{4}} & -X_{\hat{3}\hat{4}} & 0 \endbmatrix , \quad 
d_1=\bmatrix x_{11} & x_{12} & x_{13} & x_{14} \\ x_{21} & x_{22} & x_{23} & x_{24} \endbmatrix, $$
and $I_2(d_1)$ has grade 3.
Here, for $i < j$, $X_{\hat{i}\hat{j}}$ is the minor of $d_1$ obtained by removing the columns $i$ and $j$. Conversely, we denote by $X_{ij}$ the minor $x_{1i}x_{2j} - x_{1j}x_{2i}$. 

In the next sections, this complex will be used to study the Buchsbaum--Rim linkage of perfect modules with resolution of small format.
To briefly describe the structure maps $w^{(i)}_j$ of the complex $\FF$ we use the formulas from the previous subsection.

Setting $\{i,j,k,r \} = \{1,2,3,4 \}$, if $i < j < k$ we get $
w^{(3)}_{1}(e_i \wedge e_j \wedge e_k) = (-1)^r f_r$.
For the subsequent maps we need the following relation which can be easily verified:
\begin{equation}
    \label{sigigia}
    d_3(x_{2i}g_1 - x_{1i}g_2) =  X_{ji}f_j + X_{ki}f_k + X_{ri}f_r.
\end{equation}     
 We compute $w^{(2)}_1(e_1 \wedge e_2 \otimes f_1)$ and $w^{(2)}_1(e_1 \wedge e_2 \otimes f_3)$. All the other terms for $w^{(2)}_1$ can be deduced by permutation of the indexes. Explicitly, 
 $$ q^{(2)}_1(e_1 \wedge e_2 \otimes f_1) = X_{12}f_1 - X_{24}f_4 - X_{23}f_3,  \quad q^{(2)}_1(e_1 \wedge e_2 \otimes f_3) = X_{12}f_3 + X_{\hat{3}\hat{4}}(e_1^.e_2^.e_4) = 0. $$
 Thus by \eqref{sigigia}, $w^{(2)}_1(e_1 \wedge e_2 \otimes f_1) = x_{22}g_1-x_{12}g_2 $, and $w^{(2)}_1(e_1 \wedge e_2 \otimes f_3) = 0.$
 Next, setting $\varepsilon= e_1 \wedge e_2 \wedge e_3 \wedge e_4$, we obtain $$ q^{(1)}_1(\varepsilon) = \sum_{i=1}^4(-1)^i (x_{1i}u_1 + x_{2i}u_2)\otimes (-1)^i f_i.$$ 
 Hence $w^{(1)}_1(\varepsilon) = u_1 \otimes g_1 + u_2 \otimes g_2. $
 
 For $w^{(2)}_{2,1}$ it is again sufficient to compute only $w^{(2)}_{2,1}(\varepsilon \wedge e_1 \otimes f_1)$ and $w^{(2)}_{2,1}(\varepsilon \wedge e_1 \otimes f_2)$. By equation \eqref{lift2,21}, using that  
 $v^{(1)}_1(\varepsilon) \wedge d_1(e_i)=
 x_{2i}g_1 - x_{1i}g_2,$ we get
 $$ q^{(2)}_2(\varepsilon \wedge e_1 \otimes f_1)= f_2  \otimes (x_{22}g_1-x_{12}g_2) +f_3  \otimes (x_{23}g_1-x_{13}g_2)  
 $$
 $$   + f_4 \otimes (x_{24}g_1-x_{14}g_2)  + f_1 \otimes  (x_{21}g_1 - x_{11}g_2).      
 $$
 This term lifts to 
 $  w^{(2)}_2(\varepsilon \wedge e_1 \otimes f_1) = g_1 \wedge g_2. $ 
 Similarly,
 $$ q^{(2)}_2(\varepsilon \wedge e_1 \otimes f_2)= f_2  \otimes (x_{21}g_1-x_{11}g_2) +  f_3 \otimes  0  +  f_4 \otimes 0+ 
 $$ $$  - f_2 \otimes  (x_{21}g_1 - x_{11}g_2)= 0.
 $$
 lifts to 
 $  w^{(2)}_2(\varepsilon \wedge e_1 \otimes f_2) =   0.  $
 
 This computation shows that the maps $w^{(3)}_1$, $w^{(1)}_1$, $w^{(2)}_{2,1}$ corresponding to the highest graded components of the critical representations are invertible.
 These structure maps can be used to characterize whether an arbitrary $R$-module with minimal free resolution of format $(2,4,4,2)$ is perfect.
 In particular, the next theorem shows that the invertibility of these three maps is a necessary and sufficient condition for a module with minimal free resolution of format $(2,4,4,2)$ to be perfect.
 \begin{thm}
 Let $R$ be a local Gorenstein ring and let $M$ be an $R$-module having minimal free resolution of format $(2,4,4,2)$. The following conditions are equivalent:
 \begin{enumerate}
     \item $M$ is perfect.
     \item The maps $w^{(3)}_1$, $w^{(1)}_1$, $w^{(2)}_{2,1}$ are invertible.
 \end{enumerate}
 \end{thm}
 
 \begin{proof}
 The implication $1. \to 2.$ follows from the above computation. Thus assume condition 2. and consider the minimal free resolution of $M$
 $$\FF : 0 \longrightarrow F_3 \buildrel{d_3}\over\longrightarrow  F_2 \buildrel{d_2}\over\longrightarrow F_1 \buildrel{d_1}\over\longrightarrow F_0 \longrightarrow M. $$
 Say that $d_1$ has entries $z_{ij}$ and maximal minors $Z_{ij}$ and $d_3$ has entries $x_{ij}$ and maximal minors $X_{ij}$. It suffices to show that $I_2(d_1)= I_2(d_3)$.
Since $w^{(3)}_{1}$ is invertible, there exist bases of $F_1$ and $F_2$ such that 
 $w^{(3)}_{1}(e_i \wedge e_j \wedge e_k) = \pm f_r$. 
 Following the formulas used previously and computing $q^{(1)}_1(\varepsilon) = d_3(w^{(1)}_1(\varepsilon))$, we get $ q^{(1)}_1(\varepsilon) = \sum_{i=1}^4(-1)^i (z_{1i}u_1 + z_{2i}u_2)\otimes f_i$.
 Since $w^{(1)}_1$ is invertible, we can say that its matrix (of size $2 \times 2$) has entries $\lambda_{ij}$ such that $\lambda =\lambda_{11}\lambda_{22}-\lambda_{12}\lambda_{21}$ is a unit. Applying the differential $d_3$ to $w^{(1)}_1(\varepsilon)$ gives 
 $$ d_3(w^{(1)}_1(\varepsilon)) = d_3 (\lambda_{11}(u_1 \otimes g_1)+ \lambda_{12}(u_1 \otimes g_2) + \lambda_{21}(u_2 \otimes g_1) + \lambda_{22}(u_2 \otimes g_2) ).$$ Comparing with the expression for $q^{(1)}_1(\varepsilon)$ we obtain
 $$ \sum_{i=1}^4 z_{1i}  f_i = \lambda_{11}d_3(g_1)+ \lambda_{12}d_3(g_2), \quad   \sum_{i=1}^4 z_{2i}  f_i = \lambda_{21}d_3(g_1)+ \lambda_{22}d_3(g_2).   $$
 But we know that $d_3(g_h)= \sum_{i=1}^4 x_{hi}f_i$. Thus for $i=1,2,3,4$,
 $ z_{1i}= \lambda_{11}x_{1i}+\lambda_{12}x_{2i} $ and $  z_{2i}= \lambda_{21}x_{1i}+\lambda_{22}x_{2i}. $ Computing the $2 \times 2$ minors, we get
 $$ Z_{ij}= z_{1i}z_{2j}- z_{1j}z_{2i} = (\lambda_{11}x_{1i}+\lambda_{12}x_{2i})(\lambda_{21}x_{1j}+\lambda_{22}x_{2j})-(\lambda_{11}x_{1j}+\lambda_{12}x_{2j})(\lambda_{21}x_{1i}+\lambda_{22}x_{2i}) = $$
 $$ = (\lambda_{11}\lambda_{22}-\lambda_{12}\lambda_{21}) (x_{1i}x_{2j}- x_{1j}x_{2i}) = \lambda X_{ij}.  $$
 Since $\lambda$ is a unit, $I_2(d_1)= I_2(d_3)$.
   \end{proof}

\subsection{Perfect modules with resolution of format $(2,5,4,1)$}

The critical representations for the format $(2,5,4,1)$ are:
$$ W(d_3)= F_2^* \otimes [ F_3 \oplus \bw^3 F_1 \oplus \bigwedge^5 F_1 \otimes F_1 \otimes F_3^* ], $$
$$ W(d_2)  =  F_2 \otimes[ F_1^* \oplus  \bw^2 F_1 \otimes F_3^* \oplus \bigwedge^5 F_1 \otimes S_2 F_3^*], $$
$$ W(d_1)  =    F_0^* \otimes [ F_1 \oplus \bw^4 F_1 \otimes F_3^* ]. $$
Let $M$ be a perfect $R$-module having free resolution $\FF$ of this format.
This resolution $\FF$ is the dual of the resolution of a perfect ideal of format $(1,4,5,2)$. Such perfect ideals can be always obtained as an hyperplane section of a perfect ideal of height 2 with 3 generators.
Let $A= \lbrace x_{ij} \rbrace$ be a $2 \times 3$ matrix such that the ideal of maximal minors $I_2(A)$ has height 2. Let $y$ be an element of $R$ regular modulo $I_2(A)$. Then, setting $X_{ij}$ to be the minor of $A$ relative to columns $i,j$, we can express $\FF$ as follows:
\begin{equation}
\label{2451complex}
\FF : 0 \longrightarrow F_3 \buildrel{d_3}\over\longrightarrow  F_2 \buildrel{d_2}\over\longrightarrow F_1 \buildrel{d_1}\over\longrightarrow F_0 \longrightarrow M
\end{equation}
where the differentials are
$$d_3 =\bmatrix -X_{23}   \\ X_{13}  \\ -X_{12}  \\ -y  \endbmatrix, 
d_2=\bmatrix -y & 0 & 0 & X_{23} \\  0 & -y &  0& -X_{13}\\ 0 & 0 &  -y & X_{12} \\
x_{11} & x_{12} & x_{13} & 0 \\ x_{21} & x_{22} & x_{23} & 0
\endbmatrix, 
d_1=\bmatrix x_{11} & x_{12} & x_{13} & y & 0 \\ x_{21} & x_{22} & x_{23} & 0 & y \endbmatrix. $$

Let us compute all the structure maps for this complex $\FF$. Let $i,j,k$ denote distinct indices among $1,2,3$ with $i<j$.
The maps in the first graded components give
$$ e_1^.e_2^.e_3= f_4, \quad e_i^.e_j^.e_4= x_{2j}f_i - x_{2i}f_j, \quad e_i^.e_j^.e_5= x_{1j}f_i - x_{1i}f_j, \quad e_i^.e_4^.e_5= -yf_i,   $$
$$ w^{(2)}_1(e_i \wedge e_j \otimes f_k)= (-1)^{k}g, \quad  w^{(2)}_1(e_i \wedge e_4 \otimes f_4)= x_{2i}g, \quad w^{(2)}_1(e_i \wedge e_5 \otimes f_4)= x_{1i}g,  $$
$$ w^{(2)}_1(e_4 \wedge e_5 \otimes f_k)= yg, \quad w^{(1)}_1(e_i, e_j,e_k, e_4)= -g \otimes u_1, \quad w^{(1)}_1(e_i, e_j,e_k, e_5)= -g \otimes u_2.  $$
All the other entries that we do not mention are zero. Using formula \eqref{lift3,21} to compute $q^{(3)}_{2,1}$, we obtain that the only nonzero entries are $q^{(3)}_{2,1}(\varepsilon \otimes e_i)$ with $i=1,2,3$. For $i=1$, we have
$$ q^{(3)}_{2,1}(\varepsilon \otimes e_1)= f_4 \otimes -yf_1 - [x_{22}f_1 - x_{21}f_2] \otimes [x_{13}f_1 - x_{11}f_3]+ [x_{23}f_1 - x_{21}f_3] \otimes [x_{12}f_1 - x_{11}f_2]= $$
$$ = f_1 \otimes [f_1 X_{23} - f_2 X_{13}+ f_3X_{12} - y f_4] = f_1 \otimes d_3(g).  $$
Thus, by symmetry of $e_1, e_2, e_3$ in $\FF$, $w^{(3)}_{2,1}(\varepsilon \otimes e_i) = (-1)^{i+1}f_i \otimes g.$

Finally, we use formula \eqref{lift2,22} to compute $q^{(2)}_{2,2}$, noticing that the map $w^{(3)}_{2,2}$ is identically zero for this format. After checking that $ q^{(2)}_{2,2}(\varepsilon \otimes f_h)= 0 $ for $h\neq 4$, we compute
$$  q^{(2)}_{2,2}(\varepsilon \otimes f_4)= \sum_{i=1}^3 (-1)^{i+1}[(x_{1k}f_j - x_{1j}f_k) \otimes x_{2i}g - (x_{2k}f_j - x_{2j}f_k) \otimes x_{1i}g]+  $$
$$ +f_4 \otimes yg +\frac{1}{2}f_4 \otimes 2yg = 2g \otimes -d_3(g).   $$
Thus $w^{(2)}_{2,2}(\varepsilon \otimes f_4)= -2g\otimes g.$

\section{Buchsbaum--Rim linkage}\label{sec:BRlinkage}

The aim of this section is to describe how, assuming to know the structure maps of a free resolution, certain structure maps of the free resolution of a linked module can be also computed. We define a module $M'$ to be linked to $M$ if a free resolution of $M'$ is defined as the dual of the mapping cone of the comparison map of the minimal free resolution of $M$ with a Buchsbaum--Rim subcomplex.
We refer to it as {\it Buchsbaum--Rim linkage} and, analogously to classical linkage, we define the linkage class of a module.
This study prompts a characterization of the modules with small Betti numbers (formats $(2,5,4,1)$, $(2,5,5,2)$, $(2,6,5,1)$), which are in the linkage class of a perfect module resolved by a Buchsbaum--Rim complex (we allow also the case of cyclic modules resolved non-minimally by a Buchsbaum--Rim complex).

For classical linkage of ideals the same topic has been investigated for the multiplicative structure in \cite{AKM88} and for more general structure maps in \cite{GNW}.

\subsection{The free resolution of a linked module}

Let $R$ be a Gorenstein local (or graded) ring with maximal ideal $\m$. Let $M$ be a perfect $R$-module having free resolution of length three. The minimal free resolution of $M$ is
\begin{equation}
\label{basecomplex}
\A: 0 \longrightarrow A_3 \buildrel{a_3}\over\longrightarrow  A_2 \buildrel{a_2}\over\longrightarrow A_1 \buildrel{a_1}\over\longrightarrow A_0.
\end{equation}
Set $r_i= \rank \, A_i$. We generally assume that $r_0=2$ and $r_1 \geq 5$.
Denote the entries of the matrices of $a_1, a_2, a_3$ respectively by $\lbrace x_{ij} \rbrace$, $\lbrace y_{ij} \rbrace$, $\lbrace z_{ij} \rbrace$.
As before denote the basis of $A_1, A_2, A_3$ respectively by $\lbrace e_1, \ldots, e_{r_1}  \rbrace$, $\lbrace f_1, \ldots, f_{r_2} \rbrace$, $\lbrace g_1, \ldots, g_{r_3}  \rbrace$ and the basis of $A_1^*, A_2^*, A_3^*$ by $\lbrace \epsilon_1, \ldots, \epsilon_{r_1}  \rbrace$, $\lbrace \phi_1, \ldots, \phi_{r_2} \rbrace$, $\lbrace \gamma_1, \ldots, \gamma_{r_3}  \rbrace$.
We also denote by $\langle \cdot, \cdot \rangle$ the usual evaluation of an element of a module with respect to an element of its dual. 

Let $B_1 \cong R^4$ be a free module and define a linear map $b_1: B_1 \to B_0 \cong A_0 $, represented by a $2 \times 4$ matrix, whose columns are linear combinations of the columns of $d_1$, and such that $I_2(b_1)$ has height 3. Since $I_2(a_1)$ has height 3, any generic choice of $b_1$ will satisfy this property. By the results in Section 2.2, the cokernel of $b_1$ is resolved by a Buchsbaum--Rim complex of format $(2,4,4,2)$.

Let $\BB$ be the Buchsbaum--Rim complex resolving the cokernel of $b_1$ and for $i=1,2,3$ let $\alpha_i: B_i \to A_i$ be a  map obtained by lifting the identity map $B_0 \to A_0 $. Observe that $\alpha_1$ is simply defined by the relation  $a_1\alpha_1 = b_1$. Maps $\alpha_2$, $\alpha_3$ are not unique, but we will show that they are unique after fixing a choice of the maps $w^{(3)}_1$ and $w^{(1)}_1$ for the complex $\A$. This fact is analogous to what happens for the case of ideals, where the comparison maps with the Koszul complex on the generators are described using the multiplicative structure.

The modules $A_0$ and $B_0$ will be always identified and their basis will be called $\lbrace u_1, u_2 \rbrace$.
Take basis for $B_1$ equal to $\lbrace s_1, s_2, s_3, s_4  \rbrace$, basis for $B_2$ equal to $\lbrace t_1, t_2, t_3, t_4 \rbrace$ and basis for $B_3$ equal to $ \lbrace \omega_1, \omega_2 \rbrace $.
For $i=1,2,3$, let $\tau_i$ be the isomorphism $B_i^* \to B_{3-i}$ induced by the self-dual structure of $\BB$. Define maps $\beta_i: A_i^* \to B_{3-i}$ setting $\beta_i:= \tau_i \alpha_i^*$. 

The mapping cone of the complex map $\A^* \to \BB$ defined by the maps $\beta_i$ gives a free resolution $\DD$ (not necessarily minimal) of a perfect module $M'$. We say that a module $M'$ arising in this way is \it Buchsbaum--Rim linked (BR-linked) \rm to $M$.
We have 
\begin{equation}
\label{mappingcone}
\DD: 0 \longrightarrow A_1^* \buildrel{d_3}\over\longrightarrow  A_2^* \oplus B_2 \buildrel{d_2}\over\longrightarrow A_3^* \oplus B_1 \buildrel{d_1}\over\longrightarrow R.
\end{equation}
The free modules in the complex $\DD$ will be denoted by $D_3, D_2, D_1$.
The differentials are given by the following formulas:
$$ d_1 = \bmatrix \beta_3 & b_1 \endbmatrix; \quad  d_2 = \bmatrix a_3^* & 0 \\ -\beta_2 & -b_2 \endbmatrix; \quad d_3 = \bmatrix a_2^*  \\ \beta_1  \endbmatrix. $$
We show how one can express the maps $\beta_i$ in terms of the structure maps. Since our aim is to study conditions on structure maps that show that a given module is in the linkage class of a module resolved by a Buchsbaum--Rim complex, we are interested in linking in such a way that the total Betti number (i.e. the sum of all Betti numbers of the minimal free resolution) does not increase. Thus, for simplicity we consider only minimal BR-linkage. We call a BR-linkage {\it minimal} if, up to some row and column operations on $a_1$, the matrix of $b_1$ is defined simply by taking four columns of $a_1$. We assume without loss of generality that the columns of $b_1$ are exactly the first four columns of $a_1$. Under this assumption, the map $\alpha_1$ is defined by setting $\alpha_1(s_i)= e_i$ for $i=1,2,3,4$.

\begin{prop}
\label{notation}
Let $\DD$ be obtained from $\A$ by a minimal BR-linkage such that $\alpha_1(s_i)= e_i$ for $i=1,2,3,4$. Then:
\begin{enumerate}
\item[(1)] $ \beta_1(\epsilon_k)= \left\{ \begin{array}{cc} t_k &\mbox{if } k \leq 4, \\
       0  &\mbox{ otherwise. } \\
    \end{array}\right.$ 
\item[(2)] $\beta_2(\phi_h)=  \langle e_2^.e_3^.e_4, \phi_h \rangle s_1 - \langle e_1^.e_3^.e_4, \phi_h \rangle s_2 + \langle e_1^.e_2^.e_4, \phi_h \rangle s_3 -\langle e_1^.e_2^.e_3, \phi_h \rangle s_4$.
\item[(3)] $ \beta_3(\gamma_t)= \sum_{j=1}^2 \langle w^{(1)}_1(e_1 \wedge e_2 \wedge e_3 \wedge e_4), \gamma_t \otimes u_j^* \rangle u_j. $
\end{enumerate}
\end{prop}

\begin{proof}
The isomorphism $\tau_1$ identifies $s_j^*$ with $(-1)^{j} t_j = s_{k_1}^.s_{k_2}^.s_{k_3}$ where $k_1< k_2 < k_3$ are the three indices in $ \lbrace 1,2,3,4 \rbrace $ distinct from $j$. The map $\tau_2$ is the dual of $\tau_1$ and $\tau_3$ simply identifies $\omega_j$ with $u_j^{*}$. Item (1) is clear by the definition of $\alpha_1$.

For item (2), we need to prove that $ \alpha_2(t_j)= \sum_{h=1}^{r_2} (-1)^{j} \langle  e_{k_1}^.e_{k_2}^.e_{k_3}, \phi_h \rangle f_h $ (modulo the kernel of $a_3$). Hence, we need to check that $a_2(\sum_{h=1}^{r_2} (-1)^{j} \langle  e_{k_1}^.e_{k_2}^.e_{k_3}, \phi_h \rangle f_h)= \alpha_1b_2(t_j)$. By symmetry, we choose $j=1$. Thus, by equation \eqref{lift3,1} and by definition of the Buchsbaum--Rim complex we get 
$$ \sum_{h=1}^{r_2} \langle  e_{2}^.e_{3}^.e_{4}, \phi_h \rangle a_2(f_h) = \sum_{k=1}^{r_1} \left(\sum_{h=1}^{r_2} y_{kh}\langle  e_{2}^.e_{3}^.e_{4}, \phi_h \rangle \right) e_k = X_{23}e_4 - X_{24}e_3 + X_{34}e_2 = -\alpha_1b_2(t_1). $$
Analogously, for item (3) we check that, after fixing a choice of $\alpha_2$, we have $$a_3 \left( \sum_{t=1}^{r_3} \langle  w^{(1)}_1(e_1 \wedge \ldots \wedge e_4), \gamma_t \otimes u_j^* \rangle g_t \right)= \alpha_2b_3(\omega_j).$$
This follows clearly by equation \eqref{lift1,1} and shows that $\alpha_3(\omega_j)= \sum_{t=1}^{r_3} \langle  w^{(1)}_1(e_1 \wedge \ldots \wedge e_4), \gamma_t \otimes u_j^* \rangle g_t$. Item (3) follows immediately.
\end{proof}

\begin{remark}
\label{minimallinkage}
 The formulas in Proposition \ref{notation} have some important consequences. The fact that the entries of $\beta_1$ are only zeros and ones shows that the elements $\epsilon_1, \epsilon_2, \epsilon_3, \epsilon_4 $ are redundant as basis elements of $D_3$ and can been removed to minimize the free resolution.
 
 The basis elements $t_j$ of $D_2$ can be expressed as linear combinations of the elements $\phi_h$, while the elements $\omega_l$ of $D_3$ can be expressed as linear combinations of the elements $\epsilon_k$.
 In particular, recalling that $\lbrace x_i \rbrace$, $\lbrace y_{ij} \rbrace$, $\lbrace z_{ij} \rbrace$ denote the entries of $a_1, a_2, a_3$, we get:
 $$  t_j = \sum_{h=1}^{r_2} y_{jh} \phi_h, \quad \omega_l = \sum_{k=1}^{r_1} x_{lk} \epsilon_k.  $$
This shows that the rank of $A_2$ and $D_2$ are the same and the total Betti number of $\DD$ is less or equal than the total Betti number of $\A$.
\end{remark}

We define now formally the notion of linkage class of a Buchsbaum--Rim complex. We allow in this class also non-minimal complexes that correspond to free resolutions of cyclic modules.
\begin{definition}
\label{defBRlicci}
Let $R$ be a local (or graded) Gorenstein ring with infinite residue field.
Let $M$ be a perfect module having free resolution of format $(2,r_1,r_2,r_3)$. Then $M$ is in the linkage class of a BR-complex if there exists a sequence of $R$-modules 
$$ M= M_0 \sim M_1 \sim \ldots \sim M_n   $$ such that $M_i$ is BR-linked to $M_{i-1}$ for every $i \geq 1$, and $M_n$ is resolved by a (not necessarily minimal) Buchsbaum--Rim complex of format $(2,4,4,2)$.
\end{definition}

As a consequence of Proposition \ref{notation}, observe that (working over a local ring with infinite residue field) if the map $ w^{(1)}_1$ is nonzero modulo the maximal ideal of $R$, the module $M$ can be BR-linked to a non-minimal resolution $\DD$ such that some entries of the first differential $d_1$ are units. The linked module $M'$ in this case is a cyclic module (i.e. $M'= R/I$ for $I$ a perfect ideal of height 3 in $R$). In the next lemma we show that BR-linkage of a non-minimal free resolution of a cyclic module $R/I$ with $r_0=2$ having a unit in the entries of $d_1$ is equivalent to the standard linkage of the ideal $I$. This fact will allow us to apply Definition \ref{defBRlicci} to construct a sequence of BR-linked modules also in the case some of them are cyclic modules.

\begin{lem}
\label{cyclic}
Let $I=(x_1, \ldots, x_n)$ be a perfect ideal of height 3. 
Consider a non-minimal free resolution $\A$ of $R/I$ such that 
$$ a_1= \bmatrix 1 & 0 & \ldots & 0 \\
                 0 & x_1 & \ldots & x_n \endbmatrix.  $$
Assume without loss of generality that $x_1, x_2, x_3$ is a regular sequence and let $\DD$ be the free resolution obtained as BR-linkage of $\A$ choosing $b_1$ as the submatrix on the first four columns of $a_1$. Then
$$ d_1= \bmatrix 1 & 0 & \ldots & 0 \\
0 & \chi_1 & \ldots & \chi_s \endbmatrix,  $$ 
with $(\chi_1, \ldots, \chi_s) = (x_1, x_2, x_3) : I$.
\end{lem}

\begin{proof}
To get the entries of $d_1$ we need to compute the maps $w^{(3)}_1$, $w^{(1)}_1$ of the complex $\A$. Let $\A'$ be a minimal free resolution of $I=(x_1, \ldots, x_n)$ and denote the element of the basis of $A_1'$, $A_2'$, $A_3'$ by $e_k$, $f_h$, $g_t$ with $k,h,t \geq 1 $.
We can define the complex $\A$ as direct sum of $\A'$ with a split exact complex such that for $i=1,2,3$, the basis of $A_i$ is given by the basis of $A_i'$ together with an extra element which we call respectively $e_0$, $f_0$, $g_0$. These elements satisfy the relations $a_1(e_0)= u_1$, $a_2(f_0)=0$, $a_3(g_0)=f_0$.
The maps $w^{(3)}_1(\A')$, $w^{(1)}_1(\A')$ are part of the standard multiplicative structure (see \cite{GNW}). 
Looking at $w^{(3)}_1(\A)$, we adopt the notation $e_i^.e_j^.e_k = w^{(3)}_1(\A)(e_i\wedge e_j \wedge e_k)$. Thus we have $a_2(e_0^.e_i^.e_j) = x_i e_j - x_j e_i $ and $a_2(e_1^.e_2^.e_3) = 0 $. Therefore we can choose to set $  e_0^.e_i^.e_j = w^{(3)}_1(\A')(e_i \wedge e_j)  $ and $e_1^.e_2^.e_3 = f_0 $. Working similarly for the map $w^{(1)}_1(\A)$, we get $$ w^{(1)}_1(e_0 \wedge e_1 \wedge e_2 \wedge e_3) = g_0 \otimes u_1 + \left[ \sum_{t=1}^{s-3} \langle w^{(1)}_1(\A')(e_1 \wedge e_2 \wedge e_3), \gamma_t \rangle   g_t \right] \otimes u_2. $$
The thesis now follows from standard formulas for linkage of perfect ideals of grade 3 (see \cite{AKM88}, \cite{GNW}). 
\end{proof}

Looking at the computations of structure maps done in Subsection 2.3 we derive an easy corollary regarding the format $(2,5,4,1)$.
\begin{cor}
\label{cor2541}
Let $R$ be a Gorenstein local ring with infinite residue field.
Let $M$ be a perfect module having free resolution of format $(2,5,4,1)$. Then $M$ is BR-linked to a module resolved by a (non-minimal) Buchsbaum--Rim complex of format $(2,4,4,2)$.
\end{cor}

\begin{proof}

Call $\A$ the minimal free resolution of $M$.
The coefficients of the map $w^{(1)}_1$ are only zeros and ones.
   
Since we are working over a local ring with infinite residue field, using a general position argument, after row and column operations on the first differential $a_1$, we can assume that $\langle w^{(1)}_1(e_1 \wedge e_2 \wedge e_3 \wedge e_4), \gamma_1 \otimes u_1^* \rangle = 1$ and the ideal of maximal minors of the submatrix on the first four columns of $a_1$ has height 3. Consider the BR-linkage of $\A$ with respect to this submatrix of $a_1$. Calling $\DD$ the free resolution of the linked module, we can compute $d_1$ using Proposition \ref{notation}. By standard arguments, using that $R$ is local, after row and column operations we can reduce to have $d_1$ in the form 
$$ d_1= \bmatrix 1 & 0 & \ldots & 0 \\
                 0 & x_1 & \ldots & x_4 \endbmatrix.  $$
Since the ideal of maximal minors of $d_1$ needs to have height 3, we get that $I=(x_1, \ldots, x_4)$ is a perfect ideal of height 3. Moreover, by Remark \ref{minimallinkage}, the rank of $D_3$ is 1. Hence $\DD$ can be reduced to a minimal resolution of $I$ of format $(1,3,3,1)$.
This implies that $\DD$ can be presented as a (non-minimal) Buchsbaum--Rim complex of format $(2,4,4,2)$.
\end{proof}

In the next subsection, we show how structure maps are computed on the free resolution of a BR-linked module, specializing to the case of formats $(2,6,5,1)$ and $(2,5,5,2)$. We expect the generalization of these results to hold also for larger formats.

\subsection{Linking higher structure maps for formats $(2,6,5,1)$ and $(2,5,5,2)$}

In this subsection we study the behavior with respect to BR-linkage of higher structure maps for the formats $(2,5,5,2)$, $(2,6,5,1)$. We focus on maps coming from the critical representation $W(d_1)$ and show that their properties characterize when a perfect module having minimal free resolution of one of these two formats is in the BR-linkage class of a Buchsbaum--Rim complex. We start with a crucial remark.

\begin{remark}
\label{thmweyman}
In the following, to prove the theorems describing how structure maps are transfered by linkage, we will need to check that certain relations hold among some of the higher structure maps. 
As already done in \cite{GNW}, we will use the fact that any $ \prod^{3}_{i=1} GL(F_i) $-equivariant set of relations among structure maps holds over an arbitrary acyclic complex of a given format if the same relations hold when specialized to a split exact complex of that format, and the structure maps are computed using generic liftings. For the proof of this statement see \cite[Theorem 2.1]{GNW}, \cite[Lemma 2.4]{w89}, \cite[Proposition 10.4]{W18}.
\end{remark}

All the computations of relations for higher structure maps that we need in the proofs of this section will be checked over a split exact complex later in Section 4.

Let us keep the notation of the previous subsection. In the beginning we assume that $\A$ is a free resolution of a perfect module $M$ of format $(2,r_1, r_2, r_3)$, which is either $(2,6,5,1)$ or $(2,5,5,2)$. 

The complex $\DD$ will denote a free resolution of a module $ M'$ minimally BR-linked to $M$. As in the previous section we assume that the link is such that $b_1$ is the submatrix of $a_1$ obtained taking the first four columns.
To denote structure maps on different complexes we adopt the notation $w^{(i)}_{j,k}(\A)$ and
$w^{(i)}_{j,k}(\DD)$. 
Since the linkage is minimal, as a consequence of Remark \ref{minimallinkage} the basis of the modules $D_1$, $D_2$, $D_3$ are
 $\lbrace \gamma_1, \ldots, \gamma_{r_3}, s_1, \ldots, s_4  \rbrace$, $\lbrace \phi_1, \ldots, \phi_{r_2} \rbrace$, $\lbrace \epsilon_5, \ldots, \epsilon_{r_1}  \rbrace$.
 
We start analyzing how the maps $w^{(3)}_{1}(\DD)$, $w^{(1)}_{1}(\DD)$ can be computed in terms of the higher structure maps of $\A$. In a similar way one could compute also the map $w^{(2)}_{1}(\DD)$, but we omit that computation, since we want to focus mainly on the maps in the critical representation $W(d_1)$.
We also omit the computation of the entries of $w^{(3)}_{1}(s_i \wedge s_j \wedge s_k)$  and $w^{(1)}_{1}(s_1 \wedge s_2 \wedge s_3 \wedge s_4)$ since they can be obtained by those on the Buchsbaum--Rim complex $\BB$ (see Subsection 2.2) and expressed in the basis of $D_2$, $D_3$ using the relations in Remark \ref{minimallinkage}. 
Set $\varepsilon:= e_1 \wedge e_2 \wedge e_3 \wedge e_4$, and as usual $e_i^.e_j^.e_k = w^{(3)}_1(e_i \wedge e_j \wedge e_k)$.

\begin{thm}
\label{w3,1link}
The maps $w^{(3)}_{1}: \bigwedge^3 D_1 \to D_2$ and $w^{(1)}_{1}: \bigwedge^4 D_1 \to D_0 \otimes D_3$ are described as follows. 
\begin{equation}
\label{w31ssg}
s_1^.s_2^.\gamma_1 = \sum_{h=1}^{r_2} \langle w^{(2)}_1(e_1 \wedge e_2 \otimes f_h),  \gamma_1 \rangle \phi_h.
\end{equation}
\begin{equation}
\label{w31sgg}
s_1^.\gamma_1^. \gamma_2 = \sum_{h=1}^{r_2} \langle w^{(2)}_{2,1}(\varepsilon \otimes e_1  \otimes f_h),  \gamma_1 \wedge \gamma_2 \rangle \phi_h.
\end{equation}
\begin{equation}
\label{w11sssg}
w^{(1)}_{1}(s_1 \wedge s_2 \wedge s_3 \wedge \gamma_1) = \sum_{j=1}^2 \sum_{k=5}^{r_1} \langle w^{(1)}_{1}(e_1 \wedge e_2 \wedge e_3 \wedge e_k),  \gamma_1 \otimes u_j^* \rangle (\epsilon_k \otimes u_j).
\end{equation}
\begin{equation}
\label{w11ssgg}
 w^{(1)}_{1}(s_1 \wedge s_2 \wedge \gamma_1 \wedge \gamma_2)  = 
\sum_{j=1}^2 \sum_{k=5}^{r_1} \langle w^{(1)}_{2,1}(\varepsilon \wedge e_k \otimes e_1 \wedge e_2),  \gamma_1 \wedge \gamma_2 \otimes u_j^* \rangle (\epsilon_k \otimes u_j).
\end{equation} 
Analogous formulas hold for all possible combinations of basis elements $\gamma_t$ and $s_j$.
\end{thm}

\begin{proof}
The entries of $a_1, a_2, a_3$ are denoted by $\lbrace x_{ij} \rbrace$, $\lbrace y_{ij} \rbrace$, $\lbrace z_{ij} \rbrace$. Denote by $X_{ij}$ the $2 \times 2$ minors of $d_1$ and by $\chi_{jt}:= \langle w^{(1)}_1(\varepsilon), \gamma_t \otimes u_j \rangle$ the entries of $\beta_3$.
We divide the proof in four cases. \\
\bf Case 1: $s_1^.s_2^.\gamma_1$. \rm \\
We have to show that the two sides of \eqref{w31ssg} are equal after applying the differential $d_2$ on both. By definition
$$ d_2(s_1^.s_2^.\gamma_1) =  (x_{11}x_{22} -x_{12}x_{21})\gamma_1 - (x_{11}\chi_{21} -x_{21}\chi_{11})s_2 + (x_{12}\chi_{21} -x_{22}\chi_{11})s_1. $$
Recall that $d_2(\phi_h)= \sum_{t=1}^{r_3} z_{ht} \gamma_t + \sum_{j=1}^4 (-1)^{j} \langle e_{i_1}^.e_{i_2}^.e_{i_3}, \phi_h \rangle s_j,
 $ where $ i_1 < i_2 < i_3 $ are the indices in $ \lbrace 1,2,3,4 \rbrace $ distinct from $j$.
 Thus the image of the right-hand side of \eqref{w31ssg} gives
 $$ \sum_{t=1}^{r_3} \left( \sum_{h=1}^{r_2} z_{ht} \langle w^{(2)}_1(e_1 \wedge e_2 \otimes f_h),  \gamma_1 \rangle \right) \gamma_t + \sum_{j=1}^4 (-1)^{j} \left(  \sum_{h=1}^{r_2} \langle w^{(2)}_1(e_1 \wedge e_2 \otimes f_h),  \gamma_1 \rangle \cdot \langle e_{i_1}^.e_{i_2}^.e_{i_3}, \phi_h \rangle \right) s_j.  $$
 Hence, we have to prove that $\sum_{h=1}^{r_2} z_{ht} \langle w^{(2)}_1(e_1 \wedge e_2 \otimes f_h),  \gamma_1 \rangle = \delta_{1t}X_{12}$ (where $\delta_{ij}$ is the Kronecker delta) and $$\sum_{h=1}^{r_2} \langle w^{(2)}_1(e_{1} \wedge e_{2} \otimes f_h),  \gamma_1 \rangle \cdot \langle e_{i_1}^.e_{i_2}^.e_{i_3}, \phi_h \rangle = \delta_{1j}(x_{12}\chi_{21} -x_{22}\chi_{11}) -\delta_{1j}(x_{11}\chi_{21} -x_{21}\chi_{11}).$$
 By Remark \ref{thmweyman}, it is sufficient to check these relations over a split exact complex. This is done in equations 
 (4.2)-(4.3) in Lemma \ref{splitw3,1}. 
 \\
\bf Case 2: $s_1^.\gamma_1^. \gamma_2$. \rm \\
In this case we have $$ d_2(s_1^.\gamma_1^.\gamma_2) =  (x_{11}\chi_{21} -x_{21}\chi_{11})\gamma_2 - (x_{11}\chi_{22} -x_{21}\chi_{12})\gamma_1 + (\chi_{11}\chi_{22} -\chi_{21}\chi_{12})s_1. $$
As in the previous case we reduce to check that $\sum_{h=1}^{r_2} z_{ht} \langle w^{(2)}_{2,1}(\varepsilon \otimes e_1  \otimes f_h),  \gamma_1 \wedge \gamma_2 \rangle = \delta_{2t}(x_{11}\chi_{21} -x_{21}\chi_{11}) - \delta_{1t}(x_{11}\chi_{22} -x_{21}\chi_{12}),$ and $\sum_{h=1}^{r_2}  \langle w^{(2)}_{2,1}(\varepsilon \otimes e_1  \otimes f_h),  \gamma_1 \wedge \gamma_2 \rangle \cdot \langle e_{i_1}^.e_{i_2}^.e_{i_3}, \phi_h \rangle = \delta_{1j}(\chi_{11}\chi_{22} -\chi_{21}\chi_{12}).$ This is done in equations 
(4.4)-(4.5) in Lemma \ref{splitw3,1}. \\
\bf Case 3: $w^{(1)}_{1}(s_1 \wedge s_2 \wedge s_3 \wedge \gamma_1)$. \rm \\ 
Now we have to apply $d_3 \otimes 1_{D_0}$ to both sides of equation \eqref{w11sssg}.  It is sufficient to check only one component $u_j$ with $j=1,2$.
The formula for $w^{(1)}_1$, together with those proved in the previous cases and with Remark \ref{minimallinkage}, give
$$ (d_3 \otimes 1_{D_0})\langle w^{(1)}_{1}(s_1 \wedge s_2 \wedge s_3 \wedge \gamma_1), u_j^* \rangle = x_{j1} s_2^.s_3^.\gamma_1-x_{j2} s_1^.s_3^.\gamma_1 + x_{j3} s_1^.s_2^.\gamma_1 - \chi_{j1} s_1^.s_2^.s_3= $$
$$ = \sum_{h=1}^{r_2} \left[  \langle x_{j1}w^{(2)}_1(e_2 \wedge e_3 \otimes f_h) - x_{j2}w^{(2)}_1(e_1 \wedge e_3 \otimes f_h) + x_{j3}w^{(2)}_1(e_1 \wedge e_2 \otimes f_h),  \gamma_1 \rangle - \chi_{j1} y_{4h} \right] \phi_h.  $$
Applying $d_3 \otimes 1_{D_0}$, using that $d_3(\epsilon_k)= \sum_{h=1}^{r_2} y_{kh} \phi_h$, we
 have to show that the coefficient inside the brackets is equal to $  \sum_{k=5}^{r_1} y_{kh} \langle w^{(1)}_{1}(e_1 \wedge e_2 \wedge e_3 \wedge e_k),  \gamma_1 \otimes u_j^* \rangle$. This is done over a split exact complex in equation 
 (4.6) in Lemma \ref{splitw1,1}. \\
\bf Case 4: $w^{(1)}_{1}(s_1 \wedge s_2 \wedge \gamma_1 \wedge \gamma_2)$. \rm \\
Similarly as in Case 3, we have 
$$ (d_3 \otimes 1_{D_0})\langle w^{(1)}_{1}(s_1 \wedge s_2 \wedge \gamma_1 \wedge \gamma_2), u_j^* \rangle = x_{j1} s_2^.\gamma_1^.\gamma_2-x_{j2} s_1^.\gamma_1^.\gamma_2 + \chi_{j1} s_1^.s_2^.\gamma_2 - \chi_{j2} s_1^.s_2^.\gamma_1= $$
$$ = \sum_{h=1}^{r_2} \left[  \langle x_{j1}\langle w^{(2)}_{2,1}(\varepsilon \otimes e_2  \otimes f_h),  \gamma_1 \wedge \gamma_2 \rangle - x_{j2} \langle w^{(2)}_{2,1}(\varepsilon \otimes e_1  \otimes f_h),  \gamma_1 \wedge \gamma_2 \rangle + \right. $$ 
$$ + \left. \chi_{j1} \langle w^{(2)}_1(e_1 \wedge e_2 \otimes f_h), \gamma_2 \rangle - \chi_{j2} \langle w^{(2)}_1(e_1 \wedge e_2 \otimes f_h),  \gamma_1 \rangle \right] \phi_h.  $$
We have to prove that the coefficient in the brackets has to be equal to $  \sum_{k=5}^{r_1} y_{kh} \langle w^{(1)}_{2,1}(\varepsilon \wedge e_k \otimes e_1 \wedge e_2),  \gamma_1 \wedge \gamma_2 \otimes u_j^* \rangle$. This is done in equation 
(4.7) in Lemma \ref{splitw1,1}.
\end{proof}

We now look at the maps in the second graded components of $W(d_3)$. These are necessary to compute the maps in the second graded component of $W(d_1)$. Denote by $\sigma$ the wedge product $s_1 \wedge \ldots \wedge s_4$. 

\begin{thm}
\label{w3,2link}
Some entries of the maps $w^{(3)}_{2,1}: \bigwedge^5 D_1 \otimes D_1 \to D_2 \otimes D_3$ and $w^{(3)}_{2,2}: \bigwedge^6 D_1 \to D_2 \otimes D_3$ are described as follows: 
\begin{equation}
\label{w3,2,5s}
w^{(3)}_{2,1}(\sigma \wedge \gamma_1 \otimes s_1) = \sum_{k=5}^{r_1} \sum_{h=1}^{r_2} \langle w^{(2)}_1(e_k \wedge e_1 \otimes f_h), \gamma_1 \rangle (\epsilon_k \otimes \phi_h).
\end{equation}
\begin{equation}
\label{w3,2,5g}
w^{(3)}_{2,1}(\sigma \wedge \gamma_1 \otimes \gamma_1) = \sum_{k=5}^{r_1} \sum_{h=1}^{r_2} \langle w^{(2)}_{2,2}(\varepsilon \wedge e_k \otimes f_h), \gamma_1 \otimes \gamma_1 \rangle (\epsilon_k \otimes \phi_h).
\end{equation}

\begin{equation}
\label{w3,2,6}
w^{(3)}_{2,2}(\sigma \wedge \gamma_1 \wedge \gamma_2) = \sum_{k=5}^{r_1} \sum_{h=1}^{r_2} \langle w^{(2)}_{2,1}(\varepsilon \otimes e_k \otimes f_h), \gamma_1 \wedge \gamma_2 \rangle (\epsilon_k \otimes \phi_h).
\end{equation}
The same formulas hold for all the analogous combinations of basis elements $\gamma_t$ and $s_j$.
\end{thm}

\begin{proof}
Keeping the same notation and using the same method of the proof of Theorem \ref{w3,1link}, we divide the proof in three cases. 
In each of these cases we first use formulas \eqref{lift3,21} or \eqref{lift3,22} to compute the entries of the maps $ q^{(3)}_{2,1}$ and $ q^{(3)}_{2,2} $ corresponding to an arbitrary basis element $\phi_h \cdot \phi_r $ of $S_2F_2$. Then we apply the map $F_3 \otimes F_2 \to S_2F_2$ (induced by $d_3$) to the right-hand side of each of the above equations and look at the coefficient of $\phi_h \cdot \phi_r $ in the expression. In this way we identify relations among the structure maps of the complex $\A$. Such relations will be proved to hold true over a split exact complex in Section \ref{sec:splitex}. This theorem will follow then as consequence of Remark \ref{thmweyman}. \\
\bf Case 1: $w^{(3)}_{2,1}(\sigma \wedge \gamma_1 \otimes s_1)$. \rm \\
We have to compute $q^{(3)}_{2,1}(\sigma \wedge \gamma_1 \otimes s_1) = s_1^.s_2^.s_3 \otimes s_1^.s_4^.\gamma_1 - s_1^.s_2^.s_4 \otimes s_1^.s_3^.\gamma_1 + s_1^.s_3^.s_4 \otimes s_1^.s_2^.\gamma_1$. Recall that $s_{i_1}^.s_{i_2}^.s_{i_3} = (-1)^{i_1+i_2+i_3}t_{i_4}$.
By Theorem \ref{w3,1link} and Remark \ref{minimallinkage},
the coefficient of $\phi_h \cdot \phi_r $ is $$ \sum_{j=2}^4 y_{jh} \langle w^{(2)}_1(e_1 \wedge e_j \otimes f_r), \gamma_1 \rangle+ y_{jr} \langle w^{(2)}_1(e_1 \wedge e_j \otimes f_h), \gamma_1 \rangle. $$
By equation 
(4.8) in Lemma \ref{splitw3,2}, this is equal to $$ \sum_{k=5}^{r_1} y_{kr} \langle w^{(2)}_1(e_1 \wedge e_k \otimes f_h), \gamma_1 \rangle + y_{kh} \langle w^{(2)}_1(e_1 \wedge e_k \otimes f_r), \gamma_1 \rangle.  $$
\bf Case 2: $w^{(3)}_{2,1}(\sigma \wedge \gamma_1 \otimes \gamma_1)$. \rm \\
Now we have $q^{(3)}_{2,1}(\sigma \wedge \gamma_1 \otimes \gamma_1) = s_1^.s_2^.\gamma_1 \otimes s_3^.s_4^.\gamma_1 - s_1^.s_3^.\gamma_1 \otimes s_2^.s_4^.\gamma_1 + s_1^.s_4^.\gamma_1 \otimes s_2^.s_3^.\gamma_1$.
The coefficient of $\phi_h \cdot \phi_r $ is
$$  \sum_{1 \leq i < j \leq 4} (-1)^{i+j+1} \langle w^{(2)}_1(e_i \wedge e_j \otimes f_h), \gamma_1 \rangle \cdot \langle w^{(2)}_1(e_{\hat{i}} \wedge e_{\hat{j}} \otimes f_r), \gamma_1 \rangle. $$
By equation 
(4.9) in Lemma \ref{splitw3,2}, this is equal to $$ \sum_{k=5}^{r_1} y_{kr} \langle w^{(2)}_{2,2}(\varepsilon \wedge e_k \otimes f_h), \gamma_1 \otimes \gamma_1 \rangle + y_{kh} \langle w^{(2)}_{2,2}(\varepsilon \wedge e_k \otimes f_r), \gamma_1 \otimes \gamma_1 \rangle.  $$
\bf Case 3: $w^{(3)}_{2,2}(\sigma \wedge \gamma_1 \wedge \gamma_2)$. \rm \\ 
Consider the term $q^{(3)}_{2,2}(\sigma \wedge \gamma_1 \wedge \gamma_2) = \sum_{j=1}^4 t_j \otimes s_j^.\gamma_1^.\gamma_2 + \sum_{1 \leq i < j \leq 4} (-1)^{i+j+1} s_i^.s_j^.\gamma_1 \otimes s_{\hat{i}}^.s_{\hat{j}}^.\gamma_2. $ The coefficient of $\phi_h \cdot \phi_r $ is 
$$ \sum_{j=1}^4 y_{jh} \langle w^{(2)}_{2,1}(\varepsilon \otimes e_j  \otimes f_r),  \gamma_1 \wedge \gamma_2 \rangle - y_{jr} \langle w^{(2)}_{2,1}(\varepsilon \otimes e_j  \otimes f_h),  \gamma_1 \wedge \gamma_2 \rangle + $$ 
$$+ \sum_{1 \leq i < j \leq 4} (-1)^{i+j+1} \langle w^{(2)}_1(e_i \wedge e_j \otimes f_h), \gamma_1 \rangle \cdot \langle w^{(2)}_1(e_{\hat{i}} \wedge e_{\hat{j}} \otimes f_r), \gamma_2 \rangle+  $$
$$+ (-1)^{i+j} \langle w^{(2)}_1(e_i \wedge e_j \otimes f_r), \gamma_1 \rangle \cdot \langle w^{(2)}_1(e_{\hat{i}} \wedge e_{\hat{j}} \otimes f_h), \gamma_2 \rangle.  $$
By equation 
(4.10) in Lemma \ref{splitw3,2}, this is equal to $ \sum_{k=5}^{r_1} y_{kh} \langle w^{(2)}_{2,1}(\varepsilon \otimes e_k \otimes f_r), \gamma_1 \wedge \gamma_2 \rangle - y_{kr} \langle w^{(2)}_{2,1}(\varepsilon \otimes e_k \otimes f_h), \gamma_1 \wedge \gamma_2 \rangle$.
\end{proof}

Before proving the main theorem we still need to describe certain entries for the maps in the second graded component of $W(d_1)$. Again $\sigma$ denotes the wedge product $s_1 \wedge \ldots \wedge s_4$ and $\varepsilon$ denotes $e_1 \wedge \ldots \wedge e_4$.

\begin{thm}
\label{w1,2link}
Assume $\A$ to be of format $(2,6,5,1)$ and $\DD$ of format $(2,5,5,2)$, then some entries of the map $w^{(1)}_{2,1}: \bigwedge^5 D_1 \otimes \bw^2 D_1 \to D_0 \otimes \bw^2 D_3$ are described as follows: 
\begin{equation}
\label{w1,2,5}
w^{(1)}_{2,1}(\sigma \wedge \gamma_1 \otimes s_1 \wedge \gamma_1) = \sum_{j=1}^{2} \langle w^{(1)}_{2,2}(\varepsilon \wedge e_5 \wedge e_6 \otimes e_1), u_j^{*} \otimes \gamma_1 \otimes \gamma_1 \rangle (u_j \otimes \epsilon_5 \wedge \epsilon_6).
\end{equation}
Assume $\A$ to be of format $(2,5,5,2)$ and $\DD$ of format $(2,6,5,1)$, then some entries of the map $w^{(1)}_{2,2}: \bigwedge^6 D_1 \otimes D_1 \to D_0 \otimes S_2D_3$ are described as follows: 
\begin{equation}
\label{w1,2,6}
w^{(1)}_{2,2}(\sigma \wedge \gamma_1 \wedge \gamma_2 \otimes \gamma_1) = \sum_{j=1}^{2} \langle w^{(1)}_{3}(\varepsilon \wedge e_5 \otimes \varepsilon \wedge e_5), u_j^{*} \otimes \gamma_1 \wedge \gamma_2 \otimes \gamma_1 \rangle (u_j \otimes \epsilon_5 \cdot \epsilon_5).
\end{equation}

The same formulas hold for all the analogous combinations of basis elements $\gamma_t$ and $s_j$.
\end{thm}

\begin{proof}
We keep the same notation and use the same method of the proofs of Theorems \ref{w3,1link} and \ref{w3,2link}. \\
\bf Case 1: $w^{(1)}_{2,1}(\sigma \wedge \gamma_1 \otimes s_1 \wedge \gamma_1)$. \rm \\
By the formula \eqref{lift1,21}, the term $q^{(1)}_{2,1}(\sigma \wedge \gamma_1 \otimes s_1 \wedge \gamma_1) \in F_0 \otimes F_2 \otimes F_3$ is equal to 
$$ s_1^.s_2^.\gamma_1 \otimes w^{(1)}_1(s_1 \wedge s_3 \wedge s_4 \wedge \gamma_1) - s_1^.s_3^.\gamma_1 \otimes w^{(1)}_1(s_1 \wedge s_2 \wedge s_4 \wedge \gamma_1) + s_1^.s_4^.\gamma_1 \otimes w^{(1)}_1(s_1 \wedge s_2 \wedge s_3 \wedge \gamma_1) +   $$ 
$$ + d_1(s_1) w^{(3)}_{2,1}(\sigma \wedge \gamma_1 \otimes \gamma_1) -  d_1(\gamma_1) w^{(3)}_{2,1}(\sigma \wedge \gamma_1 \otimes s_1).  $$
We look at the coefficient of the basis element $u_j \otimes f_h \otimes \epsilon_k $ with $k=5,6$. Using Theorems \ref{w3,1link} and \ref{w3,2link}, we obtain that this coefficient is 
$$  \langle w^{(2)}_1(e_1 \wedge e_2 \otimes f_h),  \gamma_1 \rangle \cdot \langle w^{(1)}_{1}(e_1 \wedge e_3 \wedge e_4 \wedge e_k),  \gamma_1 \otimes u_j^* \rangle + $$  $$ - \langle w^{(2)}_1(e_1 \wedge e_3 \otimes f_h),  \gamma_1 \rangle \cdot \langle w^{(1)}_{1}(e_1 \wedge e_2 \wedge e_4 \wedge e_k),  \gamma_1 \otimes u_j^* \rangle   +  $$
$$ + \langle w^{(2)}_1(e_1 \wedge e_4 \otimes f_h),  \gamma_1 \rangle \cdot \langle w^{(1)}_{1}(e_1 \wedge e_2 \wedge e_3 \wedge e_k),  \gamma_1 \otimes u_j^* \rangle +    $$
$$ x_{j1}\langle w^{(2)}_{2,2}(\varepsilon \wedge e_k \otimes f_h), \gamma_1 \otimes \gamma_1 \rangle  - \chi_{j1}\langle w^{(2)}_1(e_k \wedge e_1 \otimes f_h), \gamma_1 \rangle.  $$
By equation 
(4.11) in Lemma \ref{splitw1,2}, this coefficient is equal to $  y_{ih} \langle w^{(1)}_{2,2}(\varepsilon \wedge e_5 \wedge e_6 \otimes e_1), u_j^{*} \otimes \gamma_1 \otimes \gamma_1 \rangle $ where $i \in \lbrace 5,6 \rbrace \setminus \lbrace k \rbrace$. \\
\bf Case 2: $w^{(1)}_{2,2}(\sigma \wedge \gamma_1 \wedge \gamma_2 \otimes \gamma_1)$. \rm \\
Let $i_1, i_2, i_3, i_4$ be distinct choices of the indices $1,2,3,4$.  We use the formula \eqref{lift1,22} to obtain $$ q^{(1)}_{2,2}(\sigma \wedge \gamma_1 \wedge \gamma_2 \otimes \gamma_1)= \sum_{i_1, i_2} (-1)^{i_1+i_2} s_{i_1}^.s_{i_2}^.\gamma_1 \otimes w^{(1)}_1(s_{i_3} \wedge s_{i_4} \wedge \gamma_1 \wedge \gamma_2) +      $$
$$ +\sum_{i_1} (-1)^{i_1}  s_{i_1}^.\gamma_1^.\gamma_2 \otimes w^{(1)}_1(s_{i_2} \wedge s_{i_3} \wedge s_{i_4} \wedge \gamma_1) + d_1(\gamma_1)w^{(3)}_{2,2}(\sigma \wedge \gamma_1 \wedge \gamma_2).    $$
The coefficient of the basis element $u_j \otimes f_h \otimes \epsilon_5 $ is 
$$ \sum_{i_1, i_2} (-1)^{i_1+i_2} \langle w^{(2)}_1(e_{i_1} \wedge e_{i_2} \otimes f_h),  \gamma_1 \rangle \cdot \langle w^{(1)}_{2,1}(\varepsilon \otimes e_5 \otimes e_{i_3} \wedge e_{i_4}),  \gamma_1 \wedge \gamma_2 \otimes u_j^* \rangle +      $$
$$ +\sum_{i_1} (-1)^{i_1}  \langle w^{(2)}_{2,1}(\varepsilon \otimes e_{i_1}  \otimes f_h),  \gamma_1 \wedge \gamma_2 \rangle \cdot \langle w^{(1)}_{1}(e_{i_2} \wedge e_{i_3} \wedge e_{i_4} \wedge e_5),  \gamma_1 \otimes u_j^* \rangle + $$ 
$$ -\chi_{j1}  \langle w^{(2)}_{2,1}(\varepsilon \otimes e_5 \otimes f_h), \gamma_1 \wedge \gamma_2 \rangle.    $$
By equation 
(4.12) in Lemma \ref{splitw1,2}, this coefficient is equal to $ y_{5h} \langle w^{(1)}_{3}(\varepsilon \wedge e_5 \otimes \varepsilon \wedge e_5), u_j^{*} \otimes \gamma_1 \wedge \gamma_2 \otimes \gamma_1 \rangle.   $ 
\end{proof}

We are now ready to characterize when a perfect module with resolution of format $(2,6,5,1)$ or $(2,5,5,2)$ is in the BR-linkage class of a Buchsbaum--Rim complex.

\begin{thm}
\label{BRliccimain}
Let $R$ be a local Gorenstein ring with infinite residue field containing $2$.
Let $M$ be a perfect $R$-module having minimal free resolution $\A$ of format $(2,6,5,1)$ or $(2,5,5,2)$. The following conditions are equivalent:
\begin{enumerate}
\item[(1)] $M$ is in the BR-linkage class of a (not necessarily minimal) Buchsbaum--Rim complex of format $(2,4,4,2)$.
\item[(2)] At least one map $w^{(1)}_{j,k}(\A)$ is nonzero modulo the maximal ideal of $R$.
\end{enumerate}
\end{thm}

\begin{proof}
Call $\m$ the maximal ideal of $R$.
First suppose that $w^{(1)}_1(\A)$ is nonzero modulo $\m$. Working exactly as in Corollary \ref{cor2541}, we can link to a complex $\DD$ such that 
$$ d_1= \bmatrix 1 & 0 & \ldots & 0 \\
0 & x_1 & \ldots & x_n \endbmatrix,  $$ 
and the ideal $I=(x_1, \ldots, x_n)$ is a perfect ideal of height 3. If $\A$ is of format $(2,5,5,2)$ then $n=5$ and the rank of $D_3$ is $1$, while if $\A$ is of format $(2,6,5,1)$, then $n=4$ and the rank of $D_3$ is at most $2$. It follows that $I$ is either Gorenstein or almost complete intersection (and its free resolution has format $(1,5,5,1)$, $(1,4,5,2)$ or $(1,3,3,1)$). Therefore $I$ can be linked to a complete intersection in at most 2 steps. Using Lemma \ref{cyclic}, we conclude as in Corollary \ref{cor2541}, obtaining that $\A$ is in the BR-linkage class of a Buchsbaum--Rim complex.

Now, let us assume that $w^{(1)}_1(\A)= 0$ modulo $\m$, but some other map $w^{(1)}_{j,k}(\A)$ with $j=2,3$ is nonzero modulo $\m$. Hence, by Proposition \ref{notation} any direct minimal BR-link of $M$ is not a cyclic module. If there exists some BR-linked module $M'$ with lower total Betti number than $M$, then the minimal free resolution $\DD$ of $M'$ has either format $(2,4,4,2)$ or $(2,5,4,1)$ (the rank of $D_1$ cannot be less than 4 since $M'$ is perfect of projective dimension 3). In both cases it is in the BR-linkage class of a Buchsbaum--Rim complex (see Corollary \ref{cor2541}).

Suppose that $\A$ is of format $(2,6,5,1)$ and any minimally BR-linked complex has format $(2,5,5,2)$. Then the third graded component of $W(a_1)$ is zero and necessarily the map $w^{(1)}_{2,2}(\A)$ is nonzero modulo $\m$. Using a standard general position argument we can assume that $\langle w^{(1)}_{2,2}(e_1 \wedge \ldots \wedge e_6 \otimes e_1), u_1^{*} \otimes \gamma_1 \otimes \gamma_1 \rangle = 1$. Linking with respect to the choice of the first four columns of $a_1$, by equation \eqref{w1,2,5} in Theorem \ref{w1,2link} we obtain that the map $  w^{(1)}_{2,1}(\DD)$ is nonzero modulo $\m$. Replacing $\A$ by $\DD$ we can pass to the case where $\A$ has format $(2,5,5,2)$ and any minimally BR-linked complex has format $(2,6,5,1)$. 

Now, if the map $w^{(1)}_{2,1}(\A)$ is nonzero modulo $\m$, using the same argument and equation \eqref{w11ssgg} in Theorem \ref{w3,1link}, we get that the map $w^{(1)}_{1}(\DD)$ is nonzero modulo $\m$ and we conclude applying the first part of this proof. If $w^{(1)}_{3}(\A)$ is nonzero modulo $\m$, we use equation \eqref{w1,2,6} in Theorem \ref{w1,2link} to get that also $w^{(1)}_{2,2}(\DD)$ is nonzero modulo $\m$. In this way we can replace $\A$ by $\DD$ and continue as in the previous case.

Finally, it is clear by Theorems \ref{w3,1link} and \ref{w1,2link} that if all the maps in $W(d_1)$ are zero modulo $\m$ for the complex $\A$, then the same happens for any BR-linked complex. Since the map $w^{(1)}_1$ of a Buchsbaum--Rim complex of format $(2,4,4,2)$ has unit coefficients, in this case we cannot have $\A$ in the BR-linkage class of a Buchsbaum--Rim complex.
\end{proof}

In the case of ideals it was conjectured in \cite{CVWdynkin} that, if $I$ is a perfect ideal of height 3 in a Gorenstein local ring $R$ admitting minimal free resolution of Dynkin format, then $I$ is licci. Furthermore, it is conjectured in \cite{SW21} that such perfect ideals of Dynkin type can be obtained as specialization of defining ideals of Schubert varieties, and that these ideals are the generic perfect ideals for such formats. Also it turns out that for most ideals of Dynkin type the highest graded structure maps of the critical representations of their free resolution $\FF$, computed with generic liftings by adding defect variables, define the differentials of a new complex, which refer to as $\FF^\bullet_{top}$. A positive answer to the conjecture in \cite{SW21} would imply that the complex $\FF^\bullet_{top}$ is split exact (up to a change of basis in the defect variables). Hence the highest graded structure maps in $W(d_1)$, $W(d_2)$, $W(d_3)$ should all be nonzero modulo the maximal ideal $\mathfrak{m}$ of $R$.
In \cite{GNW} the authors show that an ideal admitting minimal free resolution of format $(1,5,6,2)$ is licci if and only if some higher structure map in $W(d_2), W(d_3)$ is nonzero modulo $\frak{m}$ and give further evidence in support of the conjecture in \cite{CVWdynkin}.
They conjecture that a perfect ideal of height 3 in a Gorenstein local ring $R$ with infinite residue field is licci if and only if there exists some structure map in $W(d_1)$, which is nonzero modulo the maximal ideal $\frak{m}$ of $R$. Moreover, this should be equivalent to the statement that, for any linked ideal there exists some structure map in $W(d_i)$, with $i= 1,2,3$, which is nonzero modulo $\mathfrak{m}$.
For this reason it seems natural to expect that a similar pattern holds in the case of modules. In light of Theorem \ref{BRliccimain}, 
we expect that:
\begin{conjecture}
\label{conj1} \rm
Every perfect module having free resolution of format $(2,6,5,1)$ or $(2,5,5,2)$ is the BR-linkage class of a Buchsbaum--Rim complex.
\end{conjecture}

\subsection{Examples}

To construct examples of perfect modules resolved by complexes of format $(2,6,5,1)$ or $(2,5,5,2)$ we start by observing that any perfect exact complex of format $(2,6,5,1)$ is the dual of a free resolution of an ideal having format $(1,5,6,2)$. Perfect ideals with such Betti numbers are considered in several papers including \cite{annebrown}, \cite{CKLW}, \cite{kustin}. 

We consider here as an easy example the ideal $I=(x^2, y^2, z^2, xy, xz)$ in the power series ring $R=K[[x,y,z]]$. Let $\A$ be the dual of the minimal free resolution of $I$.
The differentials of $\A$ are 
\footnotesize
$$a_3 =\bmatrix x^2   \\ xy  \\ y^2  \\ xz \\ z^2 \endbmatrix, \,
a_2=\bmatrix -y & x & 0 & 0 & 0 \\ 
 0 & -y &  x & 0 & 0 \\ 
 -z & 0 &  0 & x & 0 \\ 
 0 & -z & 0 & y & 0 \\
 0 & 0 & 0 & -z & x \\
 0 & 0 & -z^2 & 0 & y^2 \\
 \endbmatrix, \,
a_1=\bmatrix z & 0 & -y & x & 0 & 0 \\ 0 & z^2 & 0 & -yz & -y^2 & 0 \endbmatrix. $$ 
\normalsize
Computing the higher structure maps of $\A$, we observe that the entries of $w^{(3)}_1(\A)$ are all contained in the maximal ideal of $R$, while the only entry of $w^{(1)}_1(\A)$ not in the maximal ideal is $\langle w^{(1)}_1(e_1 \wedge e_3 \wedge e_4 \wedge e_6), \gamma \otimes u_1^*\rangle=1$.

Using the results in the previous subsection, the BR-linkage of $\A$ choosing $\alpha_1(s_1)=e_1$, $\alpha_1(s_2)=e_2+e_3$, $\alpha_1(s_3)=e_4+e_5$, $\alpha_1(s_4)=e_6$ produces a complex which is a free resolution of a non-cyclic module, confirming the fact that $\A$ is in the BR-linkage class of a (non-minimal) Buchsbaum--Rim complex.

Instead, the BR-linkage of $\A$ with respect to the choice $\alpha_1(s_1)=e_1$, $\alpha_1(s_2)=e_2+e_4$, $\alpha_1(s_3)=e_3+e_6$, $\alpha_1(s_4)=e_5$ produces a perfect complex $\DD$ of format $(2,5,5,2)$.
The differentials of $\DD$ are \footnotesize
$$d_3 =\bmatrix 0 & -z \\z-y & 0 \\ x & z^2 \\ -y & x \\ 0 & -y^2 \endbmatrix, 
d_2=\bmatrix x^2 & xy & y^2 & xz & z^2 \\ 
 yx & y^2 &  0 & y(z-y) & -x \\ 
 -y^2 & 0 &  0 & 0 & z \\ 
 x(z-y) & yz & 0 & z(z-y) & 0 \\
 -z(z-y) & x & y-z & 0 & 0 \\
 \endbmatrix, 
d_1=\bmatrix 0 & z & x & -y & 0 \\ y-z & 0 & z(z-y) & x & -y^2 \endbmatrix. $$
\normalsize
The automorfism of $R$ defined by sending $x \to x$, $y \to -y$, $z \to z-y$ makes the complex $\DD$ dual to itself.  \ec

\section{Computation over a split exact complex}\label{sec:splitex}

In this section we work out the necessary formulas for the higher structure maps over a split exact complex in order to finish the proofs of the results in Section \ref{sec:BRlinkage}.

\noindent
Let us work over a commutative ring $R$ containing $ \frac{1}{2}$. 
Consider the split exact complex 
\begin{equation}
\label{basecomplex}
\FF: 0 \longrightarrow F_3 \buildrel{d_3}\over\longrightarrow  F_2 \buildrel{d_2}\over\longrightarrow F_1 \buildrel{d_1}\over\longrightarrow F_0 \cong R^2 
\end{equation}
on the free $R$-modules $F_0$, $F_1$, $F_2$, $F_3$ having bases $ \lbrace u_1, u_2 \rbrace$, $ \lbrace e_1, \ldots, e_n \rbrace$, $ \lbrace f_1, \ldots, f_{r_2} \rbrace$, $ \lbrace g_1, \ldots, g_m \rbrace$. Denote the dual basis by $ \lbrace u_1^*, u_2^* \rbrace$, $ \lbrace \epsilon_1, \ldots, \epsilon_{n} \rbrace$, $ \lbrace \phi_1, \ldots, \phi_{r_2} \rbrace$, $ \lbrace \gamma_1, \ldots, \gamma_m \rbrace$ where $r_2=n+m-1$. We assume the format of $\FF$ to be either $(2,5,5,2)$ or $(2,6,5,1).$

The differentials of $\FF$ are defined by imposing $d_1(e_{n-1})= u_1$, $d_1(e_{n})= u_2$, $d_1(e_i)= 0$, $ d_2(f_i)= e_i $ for $i < n-1$, $d_2(f_i)= 0$ for $i \geq n-1$, $d_3(g_i)= f_{i+n-2}$.

Let us construct a polynomial ring over $R$ by adding new variables, called defect variables.
These new variables are of the form $ b_{ijk}^u$ defined for any $1 \leq i, j, k \leq n$, $1 \leq u \leq m$ and satisfying the usual skew-symmetric relations in the indices $i,j,k$.

 Similarly, if the format of $\FF$ is $(2,5,5,2)$, we add also variables of the form $c_{i}^{ut}$ defined for any $1 \leq i \leq n$ and $1 \leq u,t \leq 2$, with the convention $c_i^{uu}=0$.
 These indeterminates are used to compute the maps $w^{(3)}_1$, $w^{(3)}_{2,1}$ in a generic way, expressing all possible liftings. 

From now on we denote by $v^{(i)}_{j,k}$ the map obtained over the complex $\FF$ by computing the corresponding $w^{(i)}_{j,k}$ with a generic lifting. As an example, the entry $e_1^.e_2^.e_3$ in $\FF$ can be chosen to be equal to $0 + \beta$, where $\beta$ is any element of the kernel of $d_2$. Since the kernel of $d_2$ is equal to the image of $d_3$, we set generically $e_1^.e_2^.e_3 = 0+ \sum_{u=1}^m b_{123}^u d_3(g_u)$. We do this similarly for all the other entries.

We describe some of the maps $v^{(i)}_j$ of the complex $\FF$ in order to check the relations appearing in the proofs of the theorems in Section 3. 
We list only some of the entries. By permutation of the indices with the usual sign rules one can obtain all the possible entries.
 As usual $\langle \cdot, \cdot \rangle$ is the evaluation map
and $\delta_{ij}$ denotes the Kronecker delta. 

Denote by $\varepsilon_{i_1,\ldots, i_r}$ the wedge product $e_{i_1}\wedge \ldots \wedge e_{i_r}.$
All the entries for these maps are computed using the formulas in Section 2.1.
For the maps in the first graded components we get
$$ \langle e_i^.e_j^.e_k, \phi_h  \rangle=  \left\{ \begin{array}{ccc} b^{h-n+2}_{ijk} &\mbox{if } h \geq n-1, \\
       \delta_{hi} &\mbox{ if } i< j = n-1, k=n, \\
       0  &\mbox{ otherwise. } \\
    \end{array}\right.  $$  $$     
 \langle v^{(2)}_1(e_i \wedge e_j \otimes f_h), \gamma_t  \rangle= \left\{ \begin{array}{ccc} -b^t_{ijh} &\mbox{if } h < n-1, \\
       \delta_{h-n+2,t} &\mbox{ if } h \geq i = n-1, j=n \\
       0  &\mbox{ otherwise. } \\
    \end{array}\right. $$
$$ \langle v^{(1)}_1(\varepsilon_{i_1,\ldots, i_4}), \gamma_t \otimes u_j^* \rangle =  \sum_{k=1}^4  (-1)^{k} \delta_{i_k,j+n-2} b^t_{k_1k_2k_3} \mbox{ where } \lbrace k_1, k_2, k_3 \rbrace = \lbrace i_1, i_2, i_3, i_4 \rbrace \setminus \lbrace i_k \rbrace. $$
 Let us look now at the maps in the second graded components.  
 For the format $(2,5,5,2)$, set 
 $$ B_{i_1i_2i_3,j_1j_2j_3}^{ut}:= b_{i_1i_2i_3}^u b_{j_1j_2j_3}^t -b_{i_1i_2i_3}^t b_{j_1j_2j_3}^u  $$
 and $$ P_{1}^{ut}= \frac{1}{2}[b_{123}^u b_{145}^t -b_{124}^u b_{135}^t +b_{134}^ub_{125}^t + b_{123}^t b_{145}^u -b_{124}^t b_{135}^u +b_{134}^t b_{125}^u] + (-1)^{t+1}c_1^{ut}.  $$ Define $P_{i}^{ut}$ analogously for the other indices.
 Then:
  $$ \langle v^{(3)}_{2,1}(\varepsilon_{1, \dots, n} \otimes e_i), \phi_h \otimes \gamma_u \rangle = \left\{ \begin{array}{ccc} P_i^{h-3,u}  &\mbox{if } h = 4,5; 
 \\      (-1)^{h+1} b_{ijk} \mbox{ with } j,k \in \lbrace 1,2,3 \rbrace \setminus \lbrace h \rbrace &\mbox{ if } h \leq 3, i >3; \\
 (-1)^{h+1} b_{123} &\mbox{ if } i,h \leq 3.
    \end{array}\right.
 $$

    $$ \langle v^{(2)}_{2,1}(\varepsilon_{i_1,\ldots, i_4} \otimes e_{i_5} \otimes f_h), \gamma_1 \wedge \gamma_2 \rangle = $$ $$ = \left\{ \begin{array}{cccc} \frac{1}{2}[B^{12}_{k_1i_5h,k_2k_3h}-B^{12}_{k_2i_5h,k_1k_3h}+B^{12}_{k_3i_5h,k_1k_2h}] &\mbox{if } h \leq 3; h \neq i_5; k_1, k_2, k_3 \neq h, i_5 ; \\
    \frac{1}{2}[B^{12}_{i_1i_2i_3,i_1i_4h}-B^{12}_{i_1i_2i_4,i_1i_3h}+B^{12}_{i_1i_3i_4,i_1i_2h}] +\delta_{h,i_5}c_{i_1} &\mbox{if } h \leq 3, i_5=i_1;
 \\  (-1)^{h+i_1-1}b^{6-h}_{i_1i_2i_5}     &\mbox{ if } i_3, i_4, h \in \lbrace 4,5 \rbrace; \\
 0 &\mbox{ otherwise. } 
    \end{array}\right.
 $$
 For the format $(2,6,5,1)$, we set $ P_{1, \hat{6}}= b_{123}b_{145}-b_{124}b_{135}+b_{134}b_{125},  $ using the convention $b_{ijk}:=b_{ijk}^1$, and define $P_{i, \hat{j}}$ analogously by permuting the indices. Also, set $$P= \sum_{1 \leq i < j}^{5} (-1)^{i+j+1} b_{ij6}b_{\hat{i}\hat{j}\hat{6}}.$$ 
 Then: 
 $$ \langle v^{(3)}_{2,2}(\varepsilon_{1, \ldots, 6}), \phi_h \otimes \gamma_1 \rangle = \left\{ \begin{array}{cc} \frac{1}{2}P &\mbox{if } h = 5; 
 \\      (-1)^{h+1} b_{\hat{h}\hat{5}\hat{6}} &\mbox{ if } h \leq 4.
    \end{array}\right.
 $$

The next series of lemmas describes relations over the split exact complex $\FF$ involving some of the maps $v^{(i)}_j$.
The first two provide the relations needed to complete the proof of Theorem \ref{w3,1link}.

In the following we denote the entries of $d_1$, $d_2$, $d_3$ respectively by $x_{ij}$, $y_{ij}$, $z_{ij}$ and the $2 \times 2$ minors of $d_1$ by $X_{ij}$. 
  
 \begin{lem}
\label{splitw3,1}
The following relations hold over the complex $\FF$ for any choice of indices such that $j_1, j_2, j_3 \in \lbrace i_1, i_2, i_3, i_4 \rbrace$.
\begin{equation}
\sum_{h=1}^{r_2} z_{ht} \langle v^{(2)}_1(e_i \wedge e_j \otimes f_h),  \gamma_s \rangle = \delta_{st}X_{ij}
\label{relw31,1}
\end{equation} 
\begin{equation}
\sum_{h=1}^{r_2} \langle v^{(2)}_1(e_{j_1} \wedge e_{j_2} \otimes f_h),  \gamma_t \rangle \cdot \langle e_{i_1}^.e_{i_2}^.e_{i_3}, \phi_h \rangle = \sum_{k=1}^{n}
 (\delta_{j_1k}x_{1j_2}  -\delta_{j_2k}x_{1j_1})
 \langle v^{(1)}_1(\varepsilon_{i_1,i_2,i_3,k}), \gamma_t \wedge u_2 \rangle + $$ $$ 
 (\delta_{j_2k}x_{2j_1}-\delta_{j_1k}x_{2j_2}) \langle v^{(1)}_1(\varepsilon_{i_1,i_2,i_3,k}), \gamma_t \wedge u_1 \rangle. 
 \label{relw31,2}
\end{equation}
If the format of $\FF$ is $(2,5,5,2)$, then
\begin{equation}
\label{relw31,3}
 \sum_{h=1}^{r_2} z_{ht} \langle v^{(2)}_{2,1}(\varepsilon_{i_1, \ldots, i_4} \otimes e_{i_1}  \otimes f_h),  \gamma_{1} \wedge \gamma_2 \rangle =  \delta_{2t}x_{1i_1}\langle v^{(1)}_1(\varepsilon_{i_1,\ldots,i_4}), \gamma_1 \wedge u_2 \rangle+ $$  $$  -\delta_{2t}x_{2i_1}\langle v^{(1)}_1(\varepsilon_{i_1,\ldots,i_4}), \gamma_1 \wedge u_1 \rangle - \delta_{1t}x_{1i_1}\langle v^{(1)}_1(\varepsilon_{i_1,\ldots,i_4}), \gamma_2 \wedge u_2 \rangle +\delta_{1t}x_{2i_1}\langle v^{(1)}_1(\varepsilon_{i_1,\ldots,i_4}), \gamma_2 \wedge u_1 \rangle.
\end{equation}
\begin{equation}
\label{relw31,4}
 \sum_{h=1}^{r_2}  \langle v^{(2)}_{2,1}(\varepsilon_{i_1, \ldots, i_4} \otimes e_{i_1}  \otimes f_h),  \gamma_1 \wedge \gamma_2 \rangle \cdot \langle e_{j_1}^.e_{j_2}^.e_{j_3}, \phi_h \rangle = $$ $$ = \langle v^{(1)}_1(\varepsilon_{j_1,j_2,j_3,i_1}), \gamma_1 \wedge u_1 \rangle \cdot \langle v^{(1)}_1(\varepsilon_{j_1,j_2,j_3,i_1}), \gamma_2 \wedge u_2 \rangle- \langle v^{(1)}_1(\varepsilon_{j_1,j_2,j_3,i_1}), \gamma_1 \wedge u_2 \rangle \cdot \langle v^{(1)}_1(\varepsilon_{j_1,j_2,j_3,i_1}), \gamma_2 \wedge u_1 \rangle.
\end{equation}
\end{lem} 
 \begin{proof}
 Observe that $z_{ht}=1$ if $h-n+2=t$, otherwise it is zero. The element $x_{ij}=1$ only if it is equal to $x_{1,n-1}$ or to $x_{2n}$, otherwise it is zero. We use the above formulas to compute left-hand and right-hand sides of the different equations.
 
 Equation 
 (4.2) becomes $\langle v^{(2)}_1(e_i \wedge e_j \otimes f_{t+n-2}),  \gamma_s \rangle = \delta_{st}X_{ij}$. Both sides are equal to 1 if $(i,j) = (n-1,n)$ and $t = s$. Otherwise they are both equal to zero.
 
  Regarding equation 
  (4.3), if $j_1,j_2 < n-1$, then the right-hand side is zero and the left-hand side reduces to $\sum_{h=1}^{n-2} b_{j_1j_2h}^t \cdot \langle e_{i_1}^.e_{i_2}^.e_{i_3}, \phi_h \rangle$, since $v^{(2)}_1(e_{j_1} \wedge e_{j_2} \otimes f_{h}) = 0$ for $h \geq n-1$. For simplicity assume $i_1 < i_2 < i_3$.
  For $h \leq n-2$ the term $\langle e_{i_1}^.e_{i_2}^.e_{i_3}, \phi_h \rangle $ is nonzero only if $h=i_1$, $i_2 = n-1$, $i_3 = n$. But this implies that either $j_1 = h$ or $j_2 = h$, thus in any case the left-hand side of the equation is also zero.
  Assume then that $j_1 < n-1$ and $j_2 = n-1$ (the case $j_2 = n$ is analogous). Also in this case $v^{(2)}_1(e_{j_1} \wedge e_{j_2} \otimes f_{h}) = 0$ for $h \geq n-1$. The right-hand side of the equation reduces to $ \langle v^{(1)}_1(\varepsilon_{i_1,i_2,i_3,j_1}), \gamma_t \wedge u_2 \rangle $, which is zero if and only if $i_1, i_2, i_3 \neq n$. If this happens, then also the left-hand side is zero, since $\langle e_{i_1}^.e_{i_2}^.e_{i_3}, \phi_h \rangle = 0$ for $h < n-1$.
  If instead $i_1 = n$, the right-hand side gives $\pm b^{t}_{i_2i_3j_1}$, while the left-hand side gives zero if $i_2,i_3 \neq j_2 = n-1$ and $\pm b^{t}_{j_1j_2h} = \pm b^{t}_{i_2i_3j_1}$ otherwise. But if $i_2,i_3 \neq j_2 = n-1$, then $i_2,i_3 < n-1$ and hence one of them is equal to $j_1$, showing that $ b^{t}_{i_2i_3j_1}=0$. Finally, we have to consider the case $(j_1, j_2)= (n-1, n)$ (if $j_1= j_2$ clearly both sides are zero). 
  If two indices among $i_1, i_2, i_3$ are also equal to $n-1, n$, the right-hand side is zero and the only nonzero terms in the sum on the left are those corresponding to $h= t+n-2$ and $h \in \lbrace i_1, i_2, i_3 \rbrace \setminus \lbrace n-1, n \rbrace$. This gives $-b_{i_1i_2i_3}^t + b_{i_1i_2i_3}^t=0$. If instead two of the indices $i_1, i_2, i_3$ are smaller than $n-1$, both sides are equal to $b_{i_1i_2i_3}^t$.
  
  For equation 
  (4.4), similarly as for equation 
  (4.2), the left-hand side is nonzero only if $i_1$ and another index among $i_2, i_3, i_4$ are equal to $4,5$.  If $i_1 \neq 4,5$, the right-hand side is clearly zero. Thus assuming $i_1=4$ and letting $s \in \lbrace 1,2 \rbrace \setminus \lbrace t \rbrace $, we get the right-hand side equal to $\pm \langle v^{(1)}_1(i_1, \ldots, i_4, \gamma_{s} \wedge u_2)  \rangle$, which is zero if $i_2, i_3, i_4 \neq 5$. Assuming without loss of generality $i_1=4$, $i_4=5$, both terms are equal to $\pm b_{i_1i_2i_3}^{s}$.
  
  For equation 
  (4.5), if either $ \lbrace 4,5 \rbrace \nsubseteq \lbrace i_1, i_2, i_3, i_4 \rbrace $ or 
  $i_1 \neq 4,5$ and $ \lbrace 4,5 \rbrace \nsubseteq \lbrace j_1, j_2, j_3 \rbrace $ both sides of the equation are clearly equal to zero. Suppose without loss of generality $i_2=j_2$ and $i_3=j_3$ and consider the two cases $i_1=j_1$ and $i_4=j_1$. In case $i_1=j_1$ the right-hand side of the equation is clearly zero. If $i_1 \neq 4,5$, the only case we did not consider is when $\lbrace i_2, i_3 \rbrace = \lbrace 4,5 \rbrace$. The left-hand side reduces to $ \langle v^{(2)}_{2,1}(\varepsilon_{i_1, \ldots, i_4} \otimes e_{i_1}  \otimes f_{i_1}),  \gamma_1 \wedge \gamma_2 \rangle = 0 $. If $\lbrace i_1, i_4 \rbrace = \lbrace 4,5 \rbrace$, then the left-hand side reduces to $\pm B^{12}_{i_1i_2i_3, i_1i_2i_3}=0$. If $\lbrace i_1, i_2 \rbrace = \lbrace 4,5 \rbrace$ (or similarly replacing $i_2$ by $i_3$), the left-hand side is $ \langle v^{(2)}_{2,1}(\varepsilon_{i_1, \ldots, i_4} \otimes e_{i_1}  \otimes f_{i_3}),  \gamma_1 \wedge \gamma_2 \rangle + B^{12}_{i_1i_2i_3, i_1i_3i_4} = 0. $
  The last case to consider is when $i_4=j_1$. Since we can assume $ \lbrace 4,5 \rbrace \subseteq \lbrace i_1, i_2, i_3, i_4 \rbrace $, we have up to permutation two relevant cases: $i_1= 4$, $i_4 = 5$ or $i_3= 4$, $i_4 = 5$. In the first case, both sides are equal to $ B^{12}_{i_1i_2i_3, i_2i_3i_4}$, in the second case they are both equal to $\langle v^{(2)}_{2,1}(\varepsilon_{i_1, \ldots, i_4} \otimes e_{i_1}  \otimes f_{i_2}),  \gamma_1 \wedge \gamma_2 \rangle = B^{12}_{i_1i_2i_3,i_1i_4i_2}.$
 \end{proof}

 For the next lemma we need the formulas for the map $v^{(1)}_{2,1}$ (in case the format of $\FF$ is $(2,5,5,2)$). Set $p:= 6- j$. For $i < k$ have  
  $$ \langle v^{(1)}_{2,1}(\varepsilon_{1,\ldots,5} \otimes e_i \wedge e_k), \gamma_1 \wedge \gamma_2 \otimes u_j^* \rangle = $$ $$ =  \left\{ \begin{array}{cccc} (\frac{1}{2})^{j+1}[B^{12}_{12p,345}-B^{12}_{13p,245}+B^{12}_{23p,145}] +c_{p} &\mbox{if } i=4, k=5; 
 \\  (\frac{1}{2})^{i+j}[B^{12}_{ii_1k,ii_2p}-B^{12}_{ii_2k,ii_1p}+B^{12}_{ikp,ii_1i_2}] +c_{i}   &\mbox{ if } i, i_1, i_2 \neq 4,5, \quad k =j+3; \\
(-1)^{i+k+1} B^{12}_{i i_1 k, i i_2 k} &\mbox{ if } i, i_1, i_2 \neq 4,5, \quad k = p; \\
 (-1)^{i+k+j} B^{12}_{123,ikp} & i,k \not \in 4,5.  
    \end{array}\right.
 $$
 
 \begin{lem}
\label{splitw1,1}
The following relations hold over the complex $\FF$ for any choice of indices.
\begin{equation}
\label{relw11,1}
 \sum_{k=1}^{n} y_{kh} \langle v^{(1)}_{1}(e_{i_1} \wedge e_{i_2} \wedge e_{i_3} \wedge e_k),  \gamma_t \otimes u_j^* \rangle = $$ $$ \langle x_{ji_1}v^{(2)}_1(e_{i_2} \wedge e_{i_3} \otimes f_h) - x_{ji_2}v^{(2)}_1(e_{i_1} \wedge e_{i_3} \otimes f_h) + x_{ji_3}v^{(2)}_1(e_{i_1} \wedge e_{i_2} \otimes f_h),  \gamma_t \rangle. 
\end{equation} 
If the format of $\FF$ is $(2,5,5,2)$, then
\begin{equation}
\label{relw11,2}
   y_{kh} \langle v^{(1)}_{2,1}(\varepsilon_{i_1, \ldots, i_4} \wedge e_k  \otimes e_{i_1} \wedge e_{i_2}),  \gamma_1 \wedge \gamma_2 \otimes u_j^* \rangle =  x_{ji_1}\langle v^{(2)}_{2,1}(\varepsilon_{i_1, \ldots, i_4} \otimes e_{i_2}  \otimes f_h),  \gamma_1 \wedge \gamma_2 \rangle  + $$  $$
  - x_{ji_2} \langle v^{(2)}_{2,1}(\varepsilon_{i_1, \ldots, i_4} \otimes e_{i_1}  \otimes f_h),  \gamma_1 \wedge \gamma_2 \rangle +   \langle v^{(1)}_{1}(\varepsilon_{i_1, \ldots, i_4}),  \gamma_1 \otimes u_j^* \rangle \langle v^{(2)}_1(e_{i_1} \wedge e_{i_2} \otimes f_h), \gamma_2 \rangle + $$ 
$$ - \langle v^{(1)}_{1}(\varepsilon_{i_1, \ldots, i_4}),  \gamma_2 \otimes u_j^* \rangle \langle v^{(2)}_1(e_{i_1} \wedge e_{i_2} \otimes f_h),  \gamma_1 \rangle.
 \end{equation}
\end{lem} 
 \begin{proof}
Observe that $y_{kh} = 1$ if $h=k < n-1$, otherwise it is zero.
For equation 
(4.6), first notice that if $i_1, i_2, i_3$ are not all distinct, both sides are clearly zero. If $h \geq n-1$, the left-hand side is zero and the right-hand side is also zero, since $v^{(2)}_1(e_{i_2} \wedge e_{i_3} \otimes f_h) \neq 0$ only if $\lbrace i_2, i_3 \rbrace = \lbrace n-1, n \rbrace$, but in that case $x_{ji_1} = 0$. Without loss of generality suppose $j=1$.
If $h < n-1$, the left-hand side reduces to $ \langle v^{(1)}_{1}(e_{i_1} \wedge e_{i_2} \wedge e_{i_3} \wedge e_h),  \gamma_t \otimes u_j^* \rangle $, which is nonzero only if $ n-1 \in \lbrace i_1, i_2, i_3 \rbrace$. If this does not happen, the right-hand side is also zero, since $ x_{ji_1}, x_{ji_2}, x_{ji_3}=0 $. If instead $i_3 = n-1$, both sides are equal to $\pm b^{t}_{i_1i_2h}$.

For equation 
(4.7), again suppose $j=1$. First say that $h \geq 4$. 
In this case, as before the left-hand side is zero. The right-hand side is clearly zero if $\lbrace 4,5 \rbrace \neq \lbrace i_1, i_2 \rbrace$. If $i_1= 4$, $i_2=5$, the right-hand side is $ \pm (b_{i_2i_3i_4}^{6-h} - b_{i_2i_3i_4}^{6-h})=0. $
Assume now $h < 4$. The left-hand side reduces to $ \langle v^{(1)}_{2,1}(\varepsilon_{i_1, \ldots, i_4} \wedge e_h  \otimes e_{i_1} \wedge e_{i_2}),  \gamma_1 \wedge \gamma_2 \otimes u_1^* \rangle $, which is nonzero only if $h=i_5 \neq i_1, \ldots, i_4$.  Both sides are then zero if $h \in \lbrace i_1, i_2 \rbrace$, since the expressions of $v^{(2)}_1$, $ v^{(2)}_{2,1} $ depend on $b_{i_1i_rh}$ or $b_{i_2i_rh}$. Also they are both zero if $ 4 \not \in \lbrace i_1, \ldots, i_4 \rbrace $ (for this observe that $x_{ji_1}=x_{ji_2}=0$ if $i_1, i_2 \neq 4$ and also the terms involving on $v^{(1)}_1$ are zero).
If $i_4=4$, the right-hand side gives $ \pm B^{12}_{i_1i_2i_3, i_1i_2h} $, which is zero if $h=i_3$. If $h=i_5$, this coincides with the left-hand side according to the above formula for $v^{(1)}_{2,1}$ (looking at the last two cases). The case where $i_3=4$ is analogous.
Finally assume $i_1 = 4$ (the case $i_2=4$ is analogous). If $h=i_3$ (or $h= i_4$), the left-hand side is zero and the right-hand side is $ \frac{1}{2}[B^{12}_{i_1i_2i_3, i_2i_4i_3} - B^{12}_{i_1i_2i_4, i_2i_3i_3} + B^{12}_{i_2i_3i_4, i_1i_2i_3}] + B^{12}_{i_2i_3i_4, i_1i_2i_3} = 0 $.
If $i_1 = 4$ and $h=i_5$, both sides are equal to $ c_{i_2} - \frac{1}{2}[B^{12}_{i_1i_2i_3, i_2i_4i_5} - B^{12}_{i_1i_2i_4, i_2i_3i_5} - B^{12}_{i_2i_3i_4, i_1i_2i_5}]. $
\end{proof}
 
 The next lemma deals with the quadratic relations in $W(d_2)$ needed to complete the proof of Theorem \ref{w3,2link}.
 For these we need to compute the map $v^{(2)}_{2,2}$ for the format $(2,6,5,1)$. Given a choice of distinct indices $i_1, \ldots, i_6$, we have 
  $$  \langle v^{(2)}_{2,2}(\varepsilon_{i_1, \ldots, i_5} \otimes f_h), \gamma_1 \wedge \gamma_1 \rangle  = \left\{ \begin{array}{cccc} \frac{1}{2} P  &\mbox{if }  h = i_6 \leq 4;  \\
  P_{h, \hat{i_6}} &\mbox{if } h \leq 4, h \in \lbrace i_1, \ldots, i_5 \rbrace;
 \\  (-1)^{i_1+i_2+i_3}b_{i_1i_2i_3}     &\mbox{ if } i_4, i_5 \geq 5, \, h =5; \\
 0 &\mbox{ if } h = 5, i_6= 5,6.
    \end{array}\right.
 $$
 \begin{lem}
\label{splitw3,2}
The following relations hold over the complex $\FF$ for any choice of indices.
\begin{equation}
\label{relw32,1}
\sum_{k=1}^{n} y_{kh} \langle v^{(2)}_1(e_i \wedge e_k \otimes f_r), \gamma_1 \rangle+ y_{kr} \langle v^{(2)}_1(e_i \wedge e_k \otimes f_h), \gamma_1 \rangle = 0.
 \end{equation} 
If the format of $\FF$ is $(2,6,5,1)$, then
\begin{equation}
\label{relw32,2}
\sum_{k=1}^{r_1} y_{kh} \langle v^{(2)}_{2,2}(\varepsilon_{i_1, \ldots, i_4} \wedge e_k \otimes f_r), \gamma_1 \otimes \gamma_1 \rangle + y_{kr} \langle v^{(2)}_{2,2}(\varepsilon_{i_1, \ldots, i_4} \wedge e_k \otimes f_h), \gamma_1 \otimes \gamma_1 \rangle = $$
$$= \sum_{1 \leq l < j \leq 4} (-1)^{l+j} \langle v^{(2)}_1(e_{i_l} \wedge e_{i_j} \otimes f_h), \gamma_1 \rangle \cdot \langle v^{(2)}_1(e_{\hat{i_l}} \wedge e_{\hat{i_j}} \otimes f_r), \gamma_1 \rangle.
\end{equation}

 If the format of $\FF$ is $(2,5,5,2)$, then
 \begin{equation}
\label{relw32,3}
\sum_{k=1}^{5} y_{kh} \langle v^{(2)}_{2,1}(\varepsilon_{i_1, \ldots, i_4} \otimes e_k \otimes f_r), \gamma_1 \wedge \gamma_2 \rangle - y_{kr} \langle v^{(2)}_{2,1}(\varepsilon_{i_1, \ldots, i_4} \otimes e_k \otimes f_h), \gamma_1 \wedge \gamma_2 \rangle = $$ 
$$= \sum_{1 \leq l < j \leq 4} (-1)^{l+j} \langle v^{(2)}_1(e_{i_l} \wedge e_{i_j} \otimes f_h), \gamma_1 \rangle \cdot \langle v^{(2)}_1(e_{\hat{i_l}} \wedge e_{\hat{i_j}} \otimes f_r), \gamma_2 \rangle+  $$
$$+(-1)^{l+j+1} \langle v^{(2)}_1(e_{i_l} \wedge e_{i_j} \otimes f_r), \gamma_1 \rangle \cdot \langle v^{(2)}_1(e_{\hat{i_l}} \wedge e_{\hat{i_j}} \otimes f_h), \gamma_2 \rangle.
 \end{equation}
\end{lem}  
 \begin{proof}
For equation 
(4.8), if $r,h \geq n-1$, the term is clearly zero since $y_{kh} = y_{kr} = 0$ for every $k$. If $h \geq n-1$, $r < n-1$, the term reduces to $ y_{rr} \langle v^{(2)}_1(e_i \wedge e_r \otimes f_h), \gamma_1 \rangle = 0 $ since $ \lbrace i,r \rbrace  \neq \lbrace n-1, n \rbrace. $ If $h,r < n-1$, we get $\langle v^{(2)}_1(e_i \wedge e_r \otimes f_h), \gamma_1 \rangle + \langle v^{(2)}_1(e_i \wedge e_h \otimes f_r), \gamma_1 \rangle = b^{1}_{irh} + b^{1}_{ihr}=0.$

For equation 
(4.9), if $r,h = 5$, the left-hand side is zero. The right-hand side is also zero, since in every product $\langle v^{(2)}_1(e_{i_l} \wedge e_{i_j} \otimes f_h), \gamma_1 \rangle \cdot \langle v^{(2)}_1(e_{\hat{i_l}} \wedge e_{\hat{i_j}} \otimes f_r), \gamma_1 \rangle$ at least one factor is zero. Suppose $h < 5$, $r = 5$. The left-hand side becomes equal to $\langle v^{(2)}_{2,2}(\varepsilon_{i_1, \ldots, i_4} \wedge e_h \otimes f_5), \gamma_1 \wedge \gamma_1 \rangle$. This term is zero if $h \in \lbrace i_1, \ldots, i_4 \rbrace$ or if $\lbrace 5,6 \rbrace  \nsubseteq \lbrace i_1, \ldots, i_4 \rbrace$. If $\lbrace 5,6 \rbrace  \nsubseteq \lbrace i_1, \ldots, i_4 \rbrace$, the right-hand side is clearly zero, since 
all the terms $v^{(2)}_1(e_{i_l} \wedge e_{i_j} \otimes f_5)$ are zero. If $i_3=5$, $i_4=6$, both sides are equal to $ \pm b_{i_1i_2h}. $ Suppose now both $h,r < 5$. In this case both terms become equal to $\sum_{1 \leq l < j \leq 4} (-1)^{l+j} b_{i_li_jh} b_{\hat{i_l}\hat{i_j}r} $. Indeed, notice that this term is zero if $h,r \in \lbrace i_1, \ldots, i_4 \rbrace$, it is $\pm P_{h,\hat{l} }$ if $h \in \lbrace i_1, \ldots, i_4 \rbrace$ and $l$ is the only index different from $i_1, \ldots, i_4, r$, and it is equal to $ P$ if $h,r \not \in \lbrace i_1, \ldots, i_4 \rbrace.$

For equation 
(4.10), again if $r,h \geq 4$, both sides are zero exactly as in the previous case. 
Suppose $h < 4$, $r \geq 4$, and for simplicity say that $r=4$. The left-hand side becomes equal to $\langle v^{(2)}_{2,1}(\varepsilon_{i_1, \ldots, i_4} \otimes e_h \otimes f_r), \gamma_1 \wedge \gamma_2 \rangle$,  
which is zero if $ \lbrace 4,5 \rbrace \nsubseteq \lbrace i_1, i_2, i_3, i_4 \rbrace $, and it is equal to $ \pm b^{2}_{i_1i_2h} $ if we choose $i_3=4, i_4=5$. With the same choice the only nonzero term on the right is $ \pm \langle v^{(2)}_1(e_{i_3} \wedge e_{i_4} \otimes f_r), \gamma_1 \rangle \cdot \langle v^{(2)}_1(e_{i_1} \wedge e_{i_2} \otimes f_h), \gamma_2 \rangle = \pm b^2_{i_1i_2h} $. If $ \lbrace 4,5 \rbrace \nsubseteq \lbrace i_1, i_2, i_3, i_4 \rbrace $ also the right-hand side is clearly zero, since so are all the terms $ v^{(2)}_1(e_{i_l} \wedge e_{i_j} \otimes f_r)$.\\
Assume now $h,r \leq 3$. The right-hand side is now equal to 
$\sum_{1 \leq l < j \leq 4} (-1)^{l+j} B^{12}_{i_li_jh,\hat{i_l}\hat{i_l}r}$ and coincides with the left-hand side because of the formulas for $v^{(2)}_{2,1}$ in the case $h \leq 3$. Indeed, if $h=r$ they are clearly both zero, if $h=i_1$, $r=i_2$, both terms are equal to $\pm 2B^{12}_{i_1i_2i_3, i_1i_2i_4}$, and if $h=i_1$, $r=i_5$ they are both equal to $\pm [B^{12}_{i_1i_2i_5, i_1i_3i_4} - B^{12}_{i_1i_3i_5, i_1i_2i_4} + B^{12}_{i_1i_4i_5, i_1i_2i_3} ]$.

\end{proof}

The last lemma is needed to complete the proof of Theorem \ref{w1,2link}. We need to compute the maps $v^{(1)}_{2,2}$ for the format $(2,6,5,1)$ and $v^{(1)}_{3}$ for the format $(2,5,5,2)$. Denote by $P(k)$ the polynomial obtained starting from $P$ and applying the permutation that switches $6$ and $k$.
We have:
$$  \langle v^{(1)}_{2,2}(\varepsilon_{1, \ldots, 6} \otimes e_k), \gamma_1 \wedge \gamma_1 \otimes u_j^{*} \rangle =  \left\{ \begin{array}{ccc} P_{k, \widehat{j+4}} &\mbox{if } k \neq j+4; 
 \\  \frac{1}{2} P(k)    &\mbox{ if } k = j+4. \\
    \end{array}\right.
    $$   
For the next formula, set $p=6-j$ and $s= j+4$.
 $$ \langle v^{(1)}_{3}(\varepsilon_{1, \ldots,5} \otimes \varepsilon_{1, \ldots,5}), \gamma_1 \wedge \gamma_2 \otimes \gamma_t \otimes u_j^* \rangle = c_1b^{t}_{23p}-c_2b^{t}_{13p} + c_3b^{t}_{12p} -c_pb^{t}_{123} + b^{t}_{12s}B^{12}_{13p,23p} + $$ 
 $$ - b^{t}_{13s}B^{12}_{12p,23p} + b^{t}_{23s}B^{12}_{12p,13p} + b^{t}_{145}B^{12}_{123,23p}-b^{t}_{245}B^{12}_{123,13p} + b^{t}_{345}B^{12}_{123,12p}.   $$
 
\begin{lem}
\label{splitw1,2}
The following relations hold over the complex $\FF$ for any choice of indices.
If the format of $\FF$ is $(2,6,5,1)$, then
\begin{equation}
\label{relw12,1}
y_{i_6h} \langle v^{(1)}_{2,2}(\varepsilon_{i_1, \ldots, i_6} \otimes e_{i_1}), u_j^{*} \otimes \gamma_1 \otimes \gamma_1 \rangle =  x_{ji_1} \langle v^{(2)}_{2,2}(\varepsilon_{i_1, \ldots, i_5} \otimes f_h), \gamma_1 \otimes \gamma_1 \rangle +$$
$$ \sum_{k=2}^5 (-1)^{k}  \langle v^{(2)}_1(e_{i_1} \wedge e_{i_k} \otimes f_h),  \gamma_1 \rangle \cdot \langle v^{(1)}_{1}(\varepsilon_{i_1, \ldots, \hat{i_k}, \ldots, i_5}),  \gamma_1 \otimes u_j^* \rangle.  
 \end{equation} 
If the format of $\FF$ is $(2,5,5,2)$, then
\begin{equation}
\label{relw12,2}
 y_{i_5h} \langle v^{(1)}_{3}(\varepsilon_{i_1,\ldots, i_5} \otimes \varepsilon_{i_1,\ldots, i_5}), u_j^{*} \otimes \gamma_1 \wedge \gamma_2 \otimes \gamma_1 \rangle = $$
$$ = \sum_{1 \leq k < l \leq 4} (-1)^{i_k+i_l} \langle v^{(2)}_1(e_{i_k} \wedge e_{i_l} \otimes f_h),  \gamma_1 \rangle \cdot \langle v^{(1)}_{2,1}(\varepsilon_{i_1, \ldots, i_5} \otimes \varepsilon_{i_1, \ldots, \hat{i_k}, \hat{i_l}, \ldots, i_4}),  \gamma_1 \wedge \gamma_2 \otimes u_j^* \rangle +      $$
$$ +\sum_{k=1}^5 (-1)^{k}  \langle v^{(2)}_{2,1}(\varepsilon_{i_1, \ldots, i_4} \otimes e_{i_k}  \otimes f_h),  \gamma_1 \wedge \gamma_2 \rangle \cdot \langle v^{(1)}_{1}(\varepsilon_{i_1, \ldots, \hat{i_k}, \ldots, i_5}),  \gamma_1 \otimes u_j^* \rangle. 
\end{equation}
\end{lem} 
 \begin{proof}
For equation 
(4.11), suppose $j=1$ (the case $j=2$ is analogous).
Recall that $y_{i_6h} \neq 0$ only if $i_6=h \leq 4$ and $x_{ji_1} \neq 0$ only if $i_1=5$. First assume $h=5$. The left-hand side is clearly zero and, if $i_6 = 5,6$ or $i_1 \neq 5$, also the right-hand one is clearly zero, since $w^{(2)}_1(e_{i_1} \wedge e_{i_k} \otimes f_5) \neq 0$ only if $i_1, i_k = 5,6$ and $w^{(1)}_{1}(\varepsilon_{i_1, \ldots, i_4}) \neq 0$ only if $5 \in \lbrace i_1, \ldots, i_4 \rbrace$. 
If $i_1= 5$ and $ 6 \in \lbrace i_2, \ldots, i_5 \rbrace$, without loss of generality, say that $i_6 = 1$ to get the right-hand side equal to $\pm (b_{234} - b_{234}) = 0$. 
Suppose then $h \leq 4$. If $h= i_1 \leq 4$, an easy check shows  that both terms are equal to zero. Hence suppose $h= i_2 \leq 4$ (the cases $h= i_3, i_4, i_5$ are analogous). The left-hand side is clearly zero and we have now three subcases: $i_6 = 5$, $5 \in \lbrace i_3, i_4, i_5 \rbrace$ or $i_1=5$. In the first case the right-hand side is easily seen to be zero. In the second case, assuming $5= i_5$, the right-hand side reduces to $b_{i_1i_3h}b_{i_1i_2i_4} - b_{i_1i_4h} b_{i_1i_2i_3} = 0$. If $i_1=5$, using the formula for $w^{(2)}_{2,2}(\varepsilon_{i_1, \ldots, i_5}  \otimes f_{i_2})$, we find that the right-hand side is equal to $\pm (P_{i_2, \hat{i_6}}- P_{i_2, \hat{i_6}})=0$.
The last case to consider is $h=i_6 \leq 4$. We have the two subcases $5 \in \lbrace i_2, i_3, i_4, i_5 \rbrace$ or $i_1=5$.
Assuming $i_5=5$, using the formula for $v^{(1)}_{2,2}(\varepsilon_{i_1, \ldots, i_6} \otimes e_{i_1})$, we obtain that both terms are equal to $P_{i_1, \hat{i_5}}$. If instead $i_1=5$, both terms are equal to $\langle v^{(1)}_{2,2}(\varepsilon_{1, \ldots, 6} \otimes e_{5}), u_1^{*} \otimes \gamma_1 \otimes \gamma_1 \rangle = \frac{1}{2}P(5) = \frac{1}{2}P +b_{126}b_{345} - b_{136}b_{245} + b_{146}b_{235} - b_{156}b_{234}.  $ 

For equation 
(4.12), again suppose $j=1$ and observe that the left-hand side is zero if either $h \geq 4$ or if $h \neq i_5$. If $h \geq 4$, and $ \lbrace 4,5 \rbrace \nsubseteq \lbrace i_1, \ldots, i_4 \rbrace $, also the right-hand side is clearly zero. If $h \geq 4$, $i_3=4,$ $i_4 = 5$, the right-hand side reduces to $\pm [B^{12}_{i_1i_2i_4,i_1i_2i_5}-B^{12}_{i_1i_2i_4,i_1i_2i_5}]=0$. Assume then $h \leq 3$. We write explicitly the computations in the case $h=i_5 =1$, $i_1=2$, $i_2=3$, $i_3=4$, $i_4=5$. The other cases can be obtained with the same method. 
Looking at the formula for $v^{(1)}_{3}$ with the choice $t=1$, $s=4$, $p=5$, we can rearrange the terms to express the left-hand side of our equation as
$$ c_1b^{2}_{235}-c_2b^{1}_{135} + c_3b^{1}_{125} -c_5b^{1}_{123} + b^{1}_{123}[b^{2}_{125}b^{1}_{345}-b^{2}_{135}b^{1}
_{245}+b^{2}_{235}b^{1}_{145}]+   $$
$$ -b^{1}_{125}[b^{2}_{123}b^{1}_{345}-b^{2}_{153}b^{1}_{243}+b^{2}_{253}b^{1}_{143}] + b^{1}_{135}[b^{2}_{125}b^{1}_{342}-b^{2}_{123}b^{1}_{245}+b^{2}_{235}b^{1}_{142}] +$$
$$ -b^{1}_{235}[b^{2}_{123}b^{1}_{145}-b^{2}_{135}b^{1}_{124}+b^{2}_{125}b^{1}_{134}]. $$
We now compute the right-hand side. Set $\mathcal{B}_5:=B^{12}_{125, 345}-B^{12}_{135,245}+B^{12}_{145,235}$ and define $\mathcal{B}(k)$ for $k=1,2,3$ by permutation.
From the terms of the form $v^{(2)}_{1} v^{(1)}_{2,1}$ we obtain 
$$ -b^{1}_{123} (c_5 + \frac{1}{2}\mathcal{B}_5)+ b^{1}_{125} (c_3 + \frac{1}{2}\mathcal{B}_3) -b^{1}_{135} (c_2 + \frac{1}{2}\mathcal{B}_2)+ b^{1}_{124}B^{12}_{235,135} -b^{1}_{134}B^{12}_{235,125} + b^{1}_{145}B^{12}_{235,123}.  $$
From the terms of the form $v^{(2)}_{2,1} v^{(1)}_{1}$ we obtain 
$$ -b^{1}_{123} (\frac{1}{2}\mathcal{B}_5 - B^{12}_{145,235}) + b^{1}_{125} (\frac{1}{2}\mathcal{B}_3 - B^{12}_{134,235}) -b^{1}_{135} (\frac{1}{2}\mathcal{B}_2 - B^{12}_{124,235})+ b^{1}_{235} (c_1 + \frac{1}{2}\mathcal{B}_1).  $$
Summing the two terms we get $-b^{1}_{123} (c_5 + \mathcal{B}_5) + b^{1}_{125} (c_3+ \mathcal{B}_3) -b^{1}_{135} (c_2 +\mathcal{B}_2)+ b^{1}_{235} (c_1 + \mathcal{B}_1).  $ The cancellation of half of the terms leads to the same expression as for $v^{(1)}_{3}$.
\end{proof}

\section*{Acknowledgements}
The authors are supported by the grants MAESTRO NCN -
UMO-2019/34/A/ST1/00263 - \\ Research in Commutative Algebra and
Representation Theory and NAWA POWROTY - PPN/PPO/2018/1/00013/U/00001 - Applications of Lie algebras to Commutative Algebra.
They also acknowledge support from INDAM-GNSAGA.
The authors would like to thank Jerzy Weyman and Xianglong Ni for helpful conversations about the content of this paper.

\bibliographystyle{alpha}

\end{document}